\documentclass{amsart}
\usepackage[utf8]{inputenc}
\usepackage{subcaption}
\usepackage{array}
\usepackage{amssymb,amsthm,latexsym,amsmath,graphics,setspace}
\usepackage{epsfig,epsf,amsfonts,mathtools}
\usepackage{color,multicol,colordvi,graphicx}
\usepackage{hyperref}

\newtheorem{theorem}{Theorem}[section]
\newtheorem{proposition}[theorem]{Proposition}

\newtheorem{corollary}[theorem]{Corollary}

\newtheorem{lemma}[theorem]{Lemma}

\newtheorem{remark}[theorem]{Remark}

\newcommand{\defeq}{\stackrel{\scriptscriptstyle\textup{def}}{=}}
\newcommand{\bm}[1]{\boldsymbol{#1}}
\newcommand{\E}{\bm{E}}
\newcommand{\R}{\mathbb{R}}

\newcommand{\cR}{\mathcal{R}}
\newcommand{\cD}{\mathcal{D}}
\newcommand{\cF}{\mathcal{F}}
\newcommand{\cT}{\mathcal{T}}
\newcommand{\T}{\mathbb{T}}
\renewcommand{\P}{\mathbb{P}}

\newcommand{\Xal}[1][]{ X^{\eps #1} }
\newcommand{\Val}{V^\eps}

\newcommand{\Zal}{Z^\eps}
\newcommand{\Vtilal}[1][]{\tilde V^{\eps #1}}

\newcommand\cO{\mathcal{O}}
\newcommand\eps{\varepsilon}
\newcommand{\norm}[1]{\left\lVert #1 \right\rVert}
\newcommand{\grad}{\nabla}
\newcommand{\abs}[1]{\left\lvert #1 \right\rvert}
\newcommand{\lap}{\Delta}

\newcommand{\paren}[1]{\left(#1\right)}

\newcommand{\mc}[1]{\mathcal{#1}}

\numberwithin{equation}{section}

\newcounter{SVCase}

\newcommand{\case}[1][\relax]{\smallskip\par\noindent\stepcounter{SVCase}\emph{Case \Roman{SVCase}: \ifx#1\relax\relax\else#1.\fi}}

\newcounter{SVStep}
\newcommand{\restartsteps}[1][0]{\setcounter{SVStep}{#1}}

\newcommand{\step}[1][\relax]{\smallskip\par\noindent\stepcounter{SVStep}{\bf Step \arabic{SVStep}: \ifx#1\relax\relax\else#1.\fi}}

\title[Diffusion approximation of Kinetic Langevin]{Error analysis of an acceleration corrected diffusion approximation of Langevin dynamics with background flow}
\author[Mori]{Yoichiro Mori}
\thanks{YM and CS were partially supported by the National Science Foundation (NSF) grant DMR2309034 (USA), and the Math+X award from the Simons Foundation (Award ID 234606). }
\address[Mori]{Department of Mathematics, Department of Biology, University of Pennsylvania, Philadelphia PA, USA}
\email{y1mori@sas.upenn.edu}
\author[Sintavanuruk]{Chanoknun Sintavanuruk}
\address[Sintavanuruk]{Department of Mathematics, University of Pennsylvania, Philadelphia PA, USA}
\address{Department of Physiology, Faculty of Medicine Siriraj Hospital, Mahidol University, Bangkok, Thailand}
\email{aidensin@sas.upenn.edu}
\author[Van]{Truong-Son P. Van}
\address[Van]{Ho Chi Minh City, Vietnam }
\email{truongson.vanp@gmail.com}
\date{\today}

\begin{document}

\begin{abstract}
	We consider the problem of approximating the Langevin dynamics of inertial particles being transported by a background flow.
	In particular, we study an acceleration corrected advection-diffusion approximation to the Langevin dynamics, a popular approximation in the study of turbulent transport.
	We prove error estimates in the averaging regime in which the dimensionless relaxation timescale $\eps$ is the small parameter.
	We show that for any finite time interval,
	the approximation error is of order $\cO(\eps)$ in the strong sense and $\cO(\eps^2)$ in the weak sense, whose optimality is checked against computational experiment.
	Furthermore, we present numerical evidence suggesting that this approximation also captures the long-time behavior of the Langevin dynamics.
\end{abstract}

\maketitle

\section{Introduction}
\label{sec:intro}
\subsection{Overview and main result}
Consider the following Langevin equation describing the motion of a particle subject to friction against a background flow and stochastic forcing:
\begin{subequations}
	\begin{equation}
		dX_{t}  =V_t\,dt,
	\end{equation}
	\begin{equation}
		m\,dV_t   = -\gamma \left(V_t -b(X_t,t)\right) \, dt + \sqrt{ 2 \cD }\,dW_t \,.
	\end{equation}
	\label{eq:Langevin_dim}
\end{subequations}
Here, $X_t\in \R^n$ is the position of the particle, $V_t \in \R^n$ the velocity, $m$ the mass,
$b$ the background flow field, $\gamma$ the friction coefficient, $\cD$ the constant representing the strength of the random force and $W_t$
is the standard $n$-dimensional Brownian motion.
In the context of molecular particles,
$\cD = \gamma k_B T$, where $k_B$ is the Boltzmann constant and $T$ the temperature~\cite{PavliotisStochasticProcesses2014}.

The probability density $\rho(x,v,t)$ of $(X_t, V_t)$ is given by
the forward Kolmogorov equation of \eqref{eq:Langevin_dim}---also known as kinetic Fokker-Planck (KFP) equation:
\begin{equation}
	\partial_{t} \rho + \nabla_x \cdot (v\rho) + \frac{\gamma}{m}\nabla_{v}\cdot\left(\left(b-v\right)\rho \right) - \frac{\cD}{m^2} \Delta_{v}\rho = 0 \,.
	\label{eq:KFP_dim}
\end{equation}
For ease of analysis,
we restrict the spatial region of $x$ (or equivalently $X_t$) to be $\mathbb{R}^n/(L_*\mathbb{Z})^n$, the $n$-dimensional torus of side length $L_*$.
The velocity field $v$ is in $\mathbb{R}^n$. Equation \eqref{eq:KFP_dim} is thus an evolution equation in $(x,v,t)\in \mathbb{R}^n/(L_*\mathbb{Z})^n\times \mathbb{R}^n\times (0,\infty)$.
The high dimensionality and the unboundedness of $v$ makes numerical computations challenging.
On the other hand, to obtain population-level statistics using SDE \eqref{eq:Langevin_dim},
one must generate a large number of numerical sample trajectories of \eqref{eq:Langevin_dim}.
Furthermore, under certain parametric regimes, the accurate generation of even a single trajectory can be computationally costly.
This motivates the study of approximations to the above equations.

Let us make the above equations dimensionless. Let $b_*$ denote the representative magnitude of the background velocity field $b$.
Take $T_*=L_*/b_*$ to be the representative timescale. There are two dimensionless parameters in the system:
\begin{equation}
	\varepsilon = \frac{m/\gamma}{L_*/b_*}, \quad \mu=\frac{\gamma/m}{\cD/(m^2b_*^2)}.
\end{equation}
The parameter $\varepsilon$ is the ratio between the relaxation time of particle velocity and the approximate time it takes for the particle
to traverse a distance $L_*$. The parameter $\mu$ controls the ratio between frictional and random forcing.
We let $\varepsilon$ be a small parameter and $\mu$ order $1$ with respect to $\varepsilon$.
We shall set $\mu=1$ in the sequel for notational simplicity; the case $\mu\neq 1$ can be dealt with in exactly the same way.
Equation~\eqref{eq:Langevin_dim} then becomes
\begin{subequations}
	\label{eq:Langevin}
	\begin{equation}
		\label{eq:Langevin-1}
		dX^\eps_{t} = V^\eps_t\,dt \,,
	\end{equation}
	\begin{equation}
		\label{eq:Langevin-2}
		dV^\eps_t
		= \frac{1}{\eps}( b(X^\eps_t,t)-V^\eps_t)\,dt + \sqrt{ \frac 2 \eps }\,dW_t \,.
	\end{equation}
\end{subequations}
The corresponding Fokker-Planck equation is
\begin{equation}
	\partial_{t} \rho^\eps + \nabla_x \cdot (v\rho^\eps) + \eps^{-1}\nabla_{v}\cdot(( b-v)\rho^\eps) - \eps^{-1} \Delta_{v}\rho^\eps = 0 \,.
	\label{eq:KFP}
\end{equation}
The above scaling
is known as the {\em averaging regime}.
Other parametric scalings of interest include the over-damped and under-damped regimes (see, for example \cite{PavliotisStochasticProcesses2014, CerraiFreidlinSmallMass2011, freidlinRandomPerturbationsDynamical2012, pavliotisMultiscaleMethods2008}), which we will not study here.

To obtain an approximation to the above when $\varepsilon$ is small, first consider the case when $b$ is a constant that does not depend on $x$ or $t$.
In this case, the distribution of $V^\eps_t$ will simply converge to a Gaussian $\mathcal{N}(b,I_n)$ with average $b$ and covariance matrix equal to the $n\times n$ identity matrix $I_n$.
Since $\varepsilon$ is small and thus velocity relaxation is fast, the distribution of $V^\eps_t$ is expected to look roughly like $\mathcal{N}(b,I_n)$ even when $b$ is non-constant.
Thus, on average, $V^\eps_t$ should look like $b$. We may thus expect $\bar{X}_t$, below, to serve as an approximation for $X_t^\eps$:
\begin{equation}
	\label{eq:ODE}
	\frac{d}{dt}\bar{X}_t = b(\bar{X}_t,t) \,.
\end{equation}
This is the classical averaging principle.
The error rates of this approximation are $\cO(\sqrt\eps)$ and $\cO(\eps)$ in the strong and weak senses, respectively.
Note that the averaging principle is not trivial in the sense that even as $\varepsilon\to 0$, the velocity distribution is not expected to converge to a delta mass at $v=b$,
but only to a distribution roughly centered around $b$ with finite non-zero variance.
The literature on this subject is vast and has a long history. We refer the reader to the following works and the references therein for further information~\cite{freidlinRandomPerturbationsDynamical2012,RocknerSunXieStrongWeak2019, pavliotisMultiscaleMethods2008, KhasminskiiLimitTheorem1966, KhasminskiiYinAveragingPrinciples2004, WeinanLiuVanden-EijndenAnalysisMultiscale2005}.

We may also formally obtain the above averaging principle as follows.
Multiply both sides of~\eqref{eq:Langevin-2} by $\eps$ to obtain:
\begin{equation}\label{approx_of_v}
	V^\eps_t = b(X_t^\eps,t) + \sqrt{2\eps} dW_t+\varepsilon dV^\eps_t.
\end{equation}
Using the first term as an approximation to $V^\eps_t$, we obtain \eqref{eq:ODE}.
One may attempt to improve the error rate, by using the first two terms of \eqref{approx_of_v}:
\begin{equation}
	\label{eq:Hasselman-simple}
	d Z^\eps_t = b(Z^\eps_t,t) dt + \sqrt{2\eps } dW_t\,.
\end{equation}
whose probability density function $u^\eps(x,t)$ satisfies the Fokker-Planck equation:
\begin{equation}
	\label{eq:naive}
	\partial_t u^\eps + \nabla_x\cdot( b u^\eps) - \eps \Delta_x u^\eps = 0 \,.
\end{equation}
We will show that the strong error rate of this approximation is $\cO(\eps)$,
an improvement of $\cO(\sqrt\eps)$ from~\eqref{eq:ODE} (see Corollary~\ref{cor:rate}).
However, the weak error rate is not improved by this approximation (see Table~\ref{tab:comparison}).

In this work, we study a different diffusion approximation of~\eqref{eq:Langevin}:
\begin{equation}\label{eq:advection-diffusion}
	dZ^\eps_t =
	F(\Zal_t, t) dt
	+ \sqrt{2\eps}dW_t, \quad F(x,t) = b(x,t)- \eps  (\partial_t b + D_x b \, b),
\end{equation}
where $D_x b$ denotes the $n\times n$ matrix of partial derivatives of the vector field $b$.
The Fokker-Planck equation for \eqref{eq:advection-diffusion} reads
\begin{equation}
	\partial_t u^\eps + \nabla_x \cdot \left(
	F u^\eps\right) - \eps \Delta_x u^\eps = 0 \,.
	\label{eq:PDEApprox}
\end{equation}
Note that $\partial_t b + D_x b \, b$ is the acceleration of the background flow field. The corrected background flow field $F(x,t)$ may thus be understood as the original velocity field $b$ minus acceleration times the relaxation time scale $\eps$.

The acceleration corrected diffusion approximation \eqref{eq:advection-diffusion} has been used extensively in the fields of turbulent transport and cloud physics
~\cite{ShawPARTICLETURBULENCEINTERACTIONS2003, ElperinKleeorinRogachevskiiSelfExcitationFluctuations1996, RogachevskiiIntroductionTurbulent2021}. 
Recently, a few works started to notice that the appearance of the accelerated correction, though small, 
may lead to new challenges in  homogenization and enhanced dissipation problems that come from inertial particles~\cite{CotiZelatiPavliotisHomogenizationHypocoercivity2020,RenaudVannesteDispersionInertial2020} .
However, to the best of our knowledge, rigorous quantitative treatment of the effect of this correction term remains very limited, even non-existent.

Without the noise term, this approximation
was first proposed by Maxey in his seminal work~\cite{MaxeyGravitationalSettling1987}, via formal asymptotic calculations.
This approximation was also used in \cite{PavliotisStuartZygalakisCalculatingEffective2009} to study the
effective diffusivity of inertial particles (see Eq. 4.10 in  ~\cite{PavliotisStuartZygalakisCalculatingEffective2009}).
Their derivation is through formal asymptotic calculations
of the backward Kolmogorov equation of \eqref{eq:Langevin}.

When the noise term is absent, the accuracy of \eqref{eq:advection-diffusion} as an approximation to \eqref{eq:Langevin} reduces to a problem for a deterministic ODE system. This approximation has been studied in \cite{HallerSapsisWhereInertial2008}, where it is shown that the error is $\mathcal{O}(\varepsilon^2)$. However, to the best of our knowledge, there has been no mathematical study of the error in the  stochastic case.

Let us state our main analytical result. The physical space $x$ is in the $n$-dimensional torus  $\mathbb{T}^n=\mathbb{R}^n/\mathbb{Z}^n$.
We make the following assumptions on the initial data for \eqref{eq:Langevin} and \eqref{eq:advection-diffusion} and on the smoothness of the background velocity field $b$:
\begin{subequations}
	\label{assumptions}
	\begin{equation}
		b\in C^\infty(\T^{n}\times [0,\infty) ; \R^n), \tag{H1}
		\label{H1}
	\end{equation}
	\begin{equation}
		X^\varepsilon_0 = Z^\varepsilon_0 = x \in \T^n  \,. \tag{H2}
		\label{H2}
	\end{equation}
	Define $\tilde{V}^\varepsilon_0:= V^\varepsilon_0 - b(X^\varepsilon_0,0)$.
	We then require for $p\geq 4$
	\begin{equation}
		\E \tilde{V}^\varepsilon_0  = 0,\quad\E |\tilde{V}^\varepsilon_0|^p <\infty  \,.
		\tag{H3}
		\label{H3}
	\end{equation}
\end{subequations}
In this paper, for $S=\T^n\times [0,\infty)$ or $S= \T^n$ we shall define the following norm for $f\in C^k(S; \R^d)$, $d\in \mathbb{N}$:
\begin{equation}
	\|f\|_{C^k}:= \max_{|\alpha|\leq k} \sup_{z\in S}
	\left|D^\alpha f(z)\right|,
\end{equation}
where $D^\alpha$ denotes the partial derivative corresponding to the multi-index $\alpha$.

\begin{theorem}
	\label{t:estimate}
  Fix $T >0$ and $p\geq4$.
	Let $W_t$ be a standard Brownian motion in $\R^n$,
	$(\Xal_t, \Val_t)$ the solution to the system of SDE~\eqref{eq:Langevin},
	$\Zal_t$  the solution of~\eqref{eq:advection-diffusion},
	and that they satisfy~\eqref{assumptions}.
	Suppose~\eqref{H1}-\eqref{H3} are true for $p$.
	Then, there exist constants $C_T$ and $\eps_0 >0$ depending on $T$ such that
	for $\eps < \eps_0$,
	\begin{equation}
    \sup_{t \in [0,T]}\E \left| \Xal_t - \Zal_t \right|^p \leq C_T  \eps^p.
		\label{ine:strong-est}
	\end{equation}
	Furthermore, for every $\varphi \in C^\infty(\T^n)$, there exists  $C_T >0$,
	\begin{equation}
		\label{eq:BrownianEst-upgrade}
		\sup_{t \in [0,T]}\left|\E\left(\varphi(\Xal_t) - \varphi(\Zal_t)\right) \right|
		\leq  C_T \eps^{2}\,.
	\end{equation}
\end{theorem}

We summarize the error rates of different approximations of equation~\eqref{eq:Langevin} in the following table.
The first row shows the classical result.
Results in the second and third rows are new.
Our main result is in the third row.
The second row is a corollary of the third row.

\begin{table}[h!]
	\centering
	\renewcommand{\arraystretch}{1.8}
	\begin{tabular}{|c|c|c|}
		\hline
		\textbf{Equation}                                                                   & \textbf{Strong Error} & \textbf{Weak Error} \\
		\hline
		$dX_t = b\,dt$ \cite{RocknerSunXieStrongWeak2019}                                   & $\sqrt{\varepsilon}$  & $\varepsilon$       \\
		\hline
		$dX_t = b \, dt + \sqrt{2\varepsilon}\,dW_t$ \cite{HasselmannStochasticClimate1976} & $\varepsilon$         & $\varepsilon$       \\
		\hline
		$dX_t = (b - \varepsilon (\partial_t b+D_xb\,b))\, dt + \sqrt{2\varepsilon}\,dW_t$           & $\varepsilon$         & $\varepsilon^2$     \\
		\hline
	\end{tabular}
	\caption{Comparison of Strong and weak error orders for different approximations of equation~\eqref{eq:Langevin}.}
	\label{tab:comparison}
\end{table}

In the absence of noise, as mentioned earlier, the error between \eqref{eq:Langevin} and \eqref{eq:advection-diffusion} is $\cO(\eps^2)$. This suggests that the weak error estimate \eqref{eq:BrownianEst-upgrade} is optimal. We also note that, when $b\equiv 0$, the error can be computed explicitly thus verifying optimality of both estimates \eqref{ine:strong-est} and \eqref{eq:BrownianEst-upgrade}.
Additionally, we numerically observe the convergence rate
matching those of Theorem~\ref{t:estimate} in Section~\ref{sec:numerics}.

\subsection{General Fast-Slow Systems}
\label{subsec:relevance}
Our system is an instance of a stochastic fast-slow system, which is encountered in many other areas of science. A particularly notable example is in climate science, where the fast variable describes changing weather patterns and the slow variable describes climate change ~\cite{HasselmannStochasticClimate1976, ImkellerVonStorchStochasticClimate2001, MajdaTimofeyevVandenEijndenMathematicalFramework2001}.
It is of interest then to obtain a closed stochastic equation for climate change, which corresponds to our advection-diffusion approximation for the Langevin equation.
Among many such systems, we mention a particular form studied by Bakhtin and Kifer~\cite{BakhtinKiferDiffusionApproximation2004}:
\begin{subequations}
	\label{eq:bakhtin-kifer}
	\begin{equation}
		dX^\eps_t = v(X^\eps_t,Y^\eps_t)\,dt + \sqrt{\eps} u(X^\eps_t,Y^\eps_t)\,dW_t,
	\end{equation}
	\begin{equation}
		dY^\eps_t = \frac{1}{\eps}b(X^\eps_t,Y^\eps_t)\,dt + \sqrt{\frac{1}{\eps}}a(X_t,Y_t)\,dW_t,
	\end{equation}
\end{subequations}
An important problem is determining the error rate of the following diffusion approximation proposed by Hasselmann~\cite{HasselmannStochasticClimate1976}:
\begin{equation}
	\label{eq:Hasselmann}
	d H^\eps_t = \bar B(H^\eps_t) dt + \sqrt{\eps} \sigma(H^\eps_t) dW_t \,.
\end{equation}
Here,
\begin{equation*}
	\bar{B}^j(x) := \int \left(v^j(x,y) + \frac{1}{2}\sum_i\partial_{y^j} u^i(x,y) a_{ij}(x,y)\right)\, d\mu_x(y)
\end{equation*}
where $\mu_x$ is the invariant distribution of the fast system,
and $\sigma$ is a limiting diffusion matrix defined in Theorem 2.1 of \cite{BakhtinKiferDiffusionApproximation2004}.
In the case of \eqref{eq:Langevin}, we have $\sigma=1$ and $\bar{B} = b$.

As before, classical averaging result says that
$$\sup_{0\leq t \leq T} \E|X^\eps_t - \bar X_t | \approx \cO(\sqrt{\eps})\,,$$ where $\bar X_t$ solves
\begin{equation*}
	\frac{d\bar{X}_t}{dt} = \bar{B}(\bar{X}_t)\,.
\end{equation*}

Bakhtin and Kifer \cite{BakhtinKiferDiffusionApproximation2004} showed that~\eqref{eq:Hasselmann} improves this rate to
$$\sup_{0\leq t \leq T}\E| X^\eps_t - H^\eps_t | \approx \cO(\eps^{1/2 + \delta}),$$
where $\delta \in (0,1/2(18+8n))$.
To do so, they introduced a different diffusion approximation
\begin{equation}
	\bar Z^\eps_t = \bar{X}_t + \sqrt{\eps} R_t \,,
	\label{eq:kifer}
\end{equation}
where $R_t$ is a solution of
\begin{equation*}
	d R_t = D_x \bar{B}( \bar{X}_t) R_t \, dt + \sigma(\bar{X}_t)\,dW_t \,.
\end{equation*}
They then showed that, for $T<\infty$ and any $\delta \in (0,1/2(18+8n))$,
\begin{equation*}
	\E^x \sup_{0 \leq t \leq T} \left| X^\eps_t- \bar Z^\eps_t \right| \leq C \eps^{\delta + 1/2} \,,
\end{equation*}
and
\begin{equation*}
	\E^x \sup_{0 \leq t \leq T} \left| H^\eps_t - \bar Z^\eps_t \right| \leq C \eps \,.
\end{equation*}

Finding the optimal $\delta$ was posed as a challenging open problem~\cite{BakhtinKiferDiffusionApproximation2004, kifer2004averaging, KiferStrongLimit2024}.
When restricting the setting to that of the Langevin equation, as a corollary of Theorem~\ref{t:estimate},
we are able to improve Kifer and Bahktin's result
for  \eqref{eq:Langevin} to achieve the optimal rate:
\begin{corollary}
	\label{cor:rate}
  Fix $T>0$ and $p \geq 4$.
	Let $W_t$ be a standard Brownian motion,
	$(\Xal_t, \Val_t)$ be solution to the system of SDE~\eqref{eq:Langevin}
	and
	let $H^\eps_t$ be the solution of~\eqref{eq:Hasselmann}.
	Then, there exist constants $C_T$ and $\eps_0>0$ depending on $T$
	such that for every $\eps<\eps_0$,
	\begin{equation}\label{eq:Hasselmann-rate}
    \sup_{t \in [0,T]} \E \left| \Xal_t - H^\eps_t \right|^p \leq C_T \eps^p \,.
	\end{equation}

	Furthermore, for every $\varphi \in C^\infty(\T^n)$, there exists $C_T > 0$ such that
	\begin{equation}
		\label{eq:weak-Hasselmann}
		\sup_{t \in [0,T]}\left|\E\left(\varphi(\Xal_t) - \varphi(H^\eps_t)\right) \right|
		\leq  C_T \eps\,.
	\end{equation}
\end{corollary}
\begin{remark}
    We stress that the estimate~\eqref{eq:weak-Hasselmann} is of order $\mathcal O(\eps)$ (not $\mathcal O(\eps^2)$ as in~\eqref{eq:BrownianEst-upgrade})
    because it lacks the correction term $\eps (\partial_tb + D_xbb)$ of the drift $b$.
    We refer the reader to Remark~\eqref{rem:O2} for further details.
    As discussed in Section~\ref{subsec:long-time} below, 
    the difference between~\eqref{eq:advection-diffusion} 
    and~\eqref{eq:Hasselmann} is not only
    quantitative but also, more strikingly, qualitative.
\end{remark}

It remains an open question whether our result in the specific case of the Langevin equation \eqref{eq:Langevin} can be extended to the much broader class represented in \eqref{eq:bakhtin-kifer}.
In this context, we mention that there are many other approximations that are not diffusion approximations.
The review~\cite{LucariniChekrounTheoreticalTools2023} provides a good overview of such approximations.
For example, \cite{LiWangXieHigherorderApproximations2023} considers the following approximation of \eqref{eq:bakhtin-kifer}:
\begin{subequations}
	\label{eq:LWX}
	\begin{equation}
		\begin{dcases}
			X^{1,\eps}_t = \bar X_t\,, \\
			dY^{1,\eps}_t = \frac{1}{\eps} b(X^{1,\eps}_t, Y^{1,\eps}_t) \, dt + \sqrt{\frac{2}{\eps}} a(X^{1,\eps}_t, Y^{1,\eps}_t) \, dW_t \,,
		\end{dcases}
	\end{equation}
	and for $k\geq 2$
	\begin{equation}
		\begin{dcases}
			dX^{k,\eps}_t = \bar B(X^\eps_t,t) \, dt + \left(\bar B(X^{k-1,\eps}_t)
			- b(X^{k-1,\eps}, Y^{k-1,\eps}_t)\right) \, dt \,, \\
			dY^{k,\eps}_t = \frac{1}{\eps} b(X^{k,\eps}_t, Y^{k,\eps}_t) \, dt + \sqrt{\frac{2}{\eps}} a(X^{k,\eps}_t, Y^{k,\eps}_t) \, dW_t   \,.
		\end{dcases}
	\end{equation}
\end{subequations}
This hierarchy has the advantage of being able to approximate equation~\eqref{eq:bakhtin-kifer} to an arbitrary order of accuracy, but its use may be challenging from a computational standpoint.

\subsection{Long time behavior}
\label{subsec:long-time}

Let us consider the long time behavior of \eqref{eq:KFP} when the background velocity $b$ is divergence free.
First consider the classical advection diffusion approximation \eqref{eq:naive}. This can be seen as describing the concentration of an inertia-less particle under passive advection and diffusion. The stationary solution $u_*(x)$ of this equation satisfies
\begin{equation}
	\nabla_x\cdot(b u_*)-\eps\Delta_x u_*=0.
\end{equation}
Assuming that $b$ is divergence-free and $x\in \mathbb{T}^n$, we may multiply the above by $u_*$ and integrate by parts in $x$ to immediately see that $u_*(x)\equiv \text{constant}$ is the only steady state. Thus, passive advection and diffusion in a divergence-free vector field cannot lead to a spatially non-uniform stationary concentration field.

However, the stationary distribution of \eqref{eq:KFP} (if it exists) may not necessarily be spatially uniform.
In fact, one of our initial motivations for this study was to explain the recently made experimental observations of non-uniform stationary concentration fields of passive particles under incompressible flow \cite{Guzman-LastraLowenMathijssenActiveCarpets2021}.

A sample Monte-Carlo simulation of the stationary spatial distribution for \eqref{eq:KFP} is given in Figure \ref{fig:longtime-approx}.
Note that the spatial non-uniformity is greater than three-fold between the minimum and maximum of this plot, even when $\varepsilon$ is small.
In fact, even in the limit of $\varepsilon\to 0$, the stationary distribution is not expected to become uniform.
Equation \eqref{eq:naive} will not be able to approximate this stationary distribution given our discussion above.
In the same figure, we plot the stationary distribution of the acceleration corrected approximation \eqref{eq:PDEApprox}, which is seen to provide a good approximation to the stationary distribution of \eqref{eq:KFP}.
This then demonstrates the potential utility of \eqref{eq:PDEApprox} for studying the long-time behavior of \eqref{eq:Langevin}.
This is somewhat surprising given that the difference between \eqref{eq:PDEApprox} and \eqref{eq:naive} is only a term proportional to $\eps$.

\begin{figure}[h!]
	\begin{center}
		\includegraphics[width=0.8\textwidth]{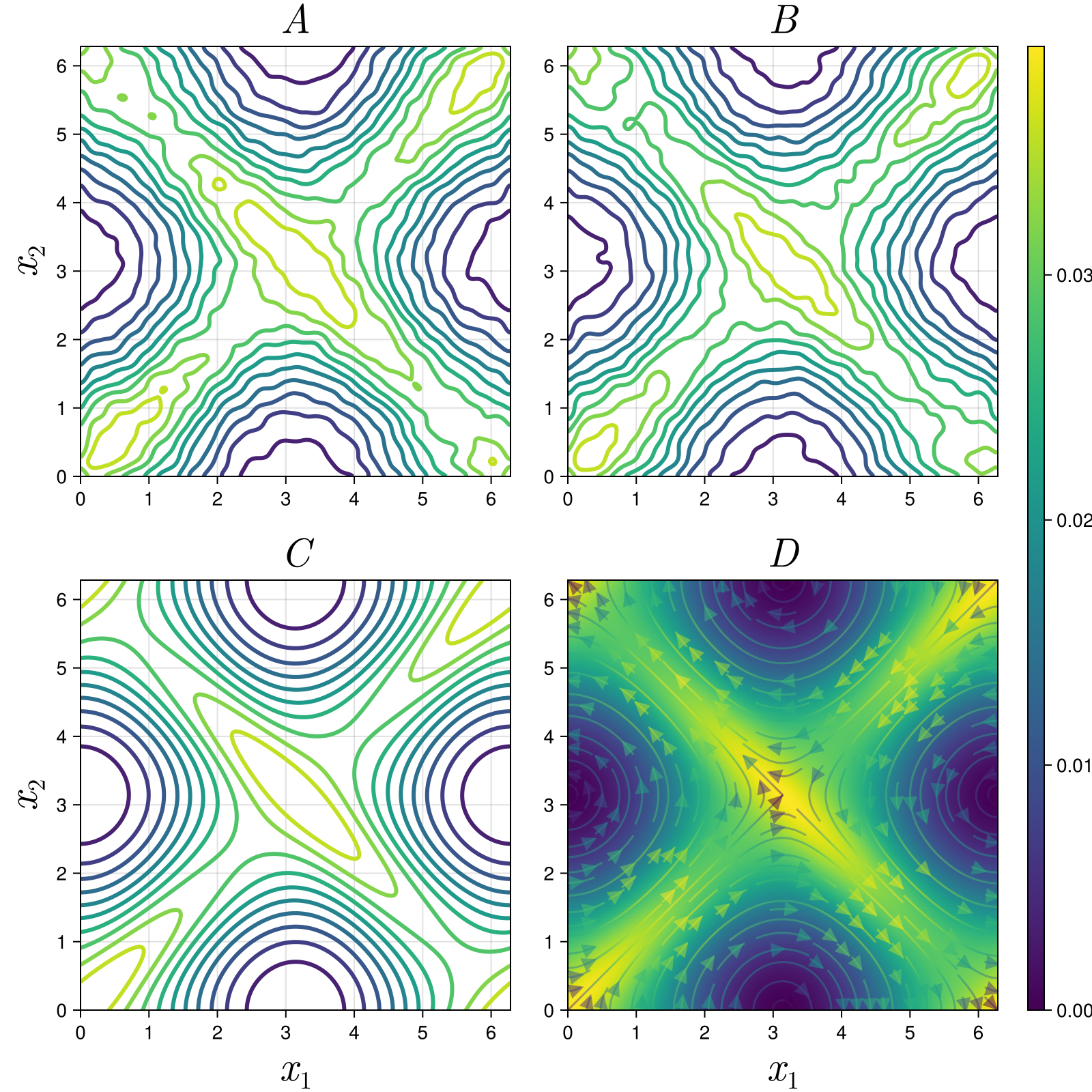}
	\end{center}
	\caption{Long-time behaviors at $T=100$.
		Row 1 displays the contours of the densities from Monte-Carlo simulation and reconstructed from kernel density estimation:
		(A) is the density of $X^\varepsilon_T$ satisfying the Langevin equation~\eqref{eq:Langevin},
		(B) is the density $Z^\varepsilon_T$ satisfying the approximation~\eqref{eq:advection-diffusion}.
		Row 2 displays the stationary solution $u^\varepsilon$ of PDE~\eqref{eq:PDEApprox}.
		(C) is the contour of $u^\eps$.
		(D) is the heatmap of $u^\eps$, together with the background flow. }
	\label{fig:longtime-approx}
\end{figure}

\begin{figure}
	\begin{center}
		\includegraphics[width=0.8\textwidth]{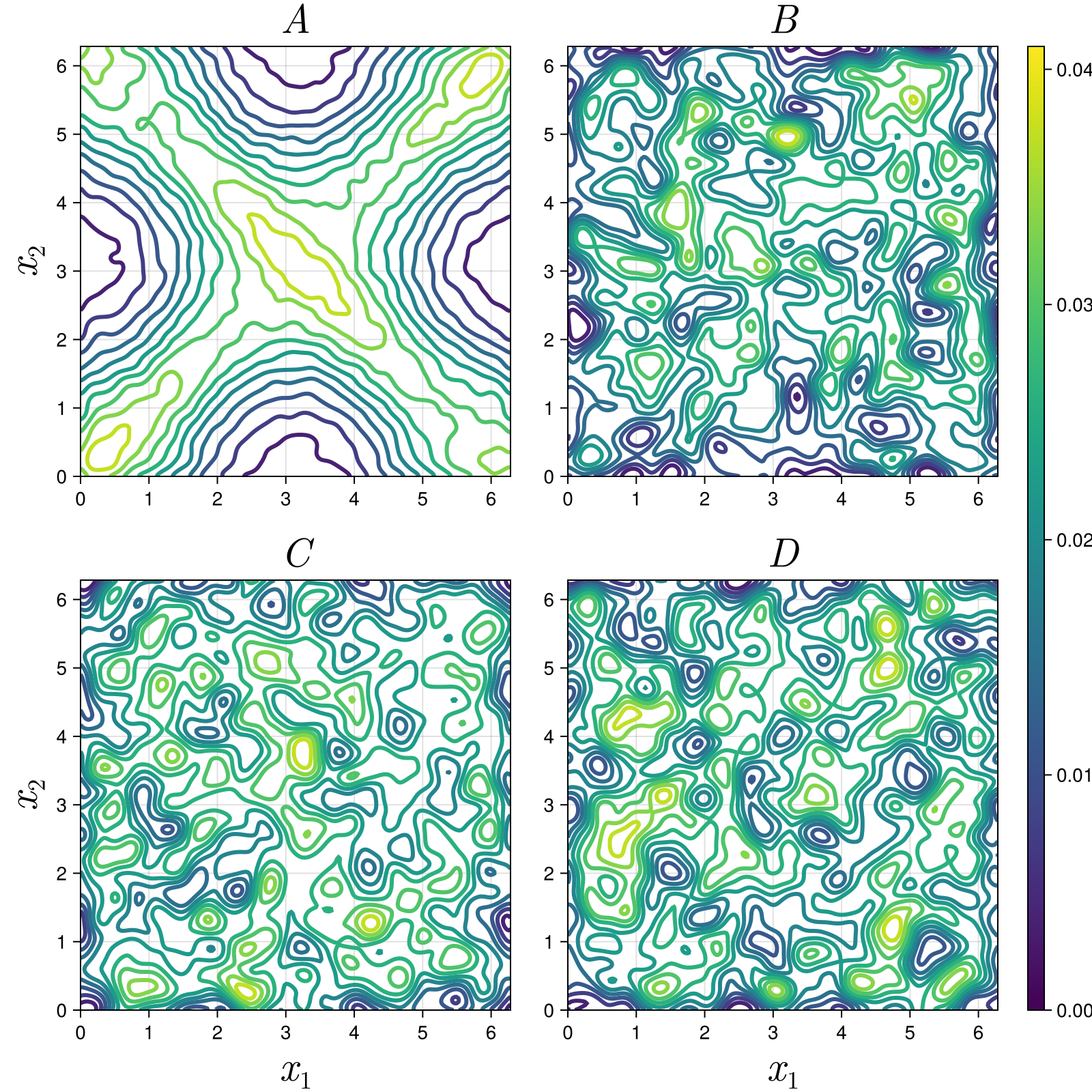}
	\end{center}
	\caption{
		Long time behavior at $T = 100$ for different methods.
		(A) describes the density of~\eqref{eq:advection-diffusion},
		(B) the density of~\eqref{eq:Hasselman-simple},
		(C) the density of~\eqref{eq:bakhtin-kifer},
		(D) the density of~\eqref{eq:LWX}.
	}
	\label{fig:prev_approx}
\end{figure}

An extensive computational or mathematical study of the long time behavior of \eqref{eq:KFP} and that of \eqref{eq:advection-diffusion} is beyond the scope of this paper. We note that equations of the form \eqref{eq:advection-diffusion} where $b$ is divergence free appear in other contexts. Equation \eqref{eq:naive} with a divergence-free vector field in 2D can be seen as the Fokker-Planck equation a stochastically forced Hamiltonian system. Equation \eqref{eq:PDEApprox} is then a non-Hamiltonian perturbation to such a system. The invariant measure of such systems is studied in \cite{WolanskyStochasticPerturbations1988}.

We finally point out that the physical reason for the above spatial non-uniformity is that inertial particles do not exactly follow the flow lines of the divergence free vector field $b$. This was indeed the key observation in \cite{MaxeyGravitationalSettling1987} who, on this basis, predicted that ``particles will tend to accumulate in regions of high strain rate or low vorticity''. This was later confirmed by physical experiment (see for example \cite{PetersenBakerColettiExperimentalStudy2019}).

Thus far, this phenomenon has been studied mathematically using a pathwise approach. In \cite{HallerSapsisWhereInertial2008}, the authors study the deterministic ODE version of \eqref{eq:advection-diffusion} as an approximation to  the deterministic ODE version of \eqref{eq:Langevin}. The dynamical systems point of view taken in these studies was also extended to the stochastic case in \cite{RenDuanJonesApproximationRandom2015}. The study of the stationary distribution can be seen as taking a population-level viewpoint as opposed to a pathwise viewpoint. An extensive study of the stationary distribution and long time behavior will be a subject of future study.

\subsection{Outline of the paper}
In Section \ref{sec:ideas}, we discuss the overall strategy and ideas behind the proof.
In Sections~\ref{sec:limiting},
we   describe some structural properties of the system that may be
quantitatively approximated by It\^o isometry coming from the a typical diffusion process.
The proof of the strong estimate~\eqref{ine:strong-est} of Theorem~\ref{t:estimate} is given in Section~\ref{sec:strong}.
The proof of the weak estimate is more involved and spans Sections~\ref{sec:A-F}-\ref{sec:proof-weak}.
In Section~\ref{sec:simulation}, we numerically verify the estimates proved in Theorem \ref{t:estimate}. We conclude with a preliminary computational study of the long time behavior of \eqref{eq:KFP} and \eqref{eq:advection-diffusion}.

\section{Proof ideas}\label{sec:ideas}

\subsection{Warm-up}
We motivate our discussion by the following simple result for the case $b\equiv 0$ in \eqref{eq:Langevin}.
This is essentially Theorem \ref{t:estimate} when $b\equiv 0$.
\begin{proposition}
	\label{prop:BrownianEst}
	Let $W_t$ be a standard Brownian motion in $\mathbb{R}^n$, $(\tilde Y^\eps_t, \Vtilal_t)$ be solution to the system of SDE
	\begin{subequations}
		\label{eq:bm-approx}
		\begin{align}
			d X^\eps_t
			         & = \Vtilal_t \, dt \,,                                        \\
			d \Vtilal_t
			         & = - \frac{\Vtilal_t}{\eps} dt  + \sqrt{\frac 2 \eps} dW_t\,, \\
			X^\eps_0 & = 0 \,.
		\end{align}
	\end{subequations}
	Let $Z^\eps_t = \sqrt{2 \eps} W_t $ and assume that $\Vtilal_0$ satisfies~\eqref{H3}.
	Then, for every $T>0$ and $\varphi \in C^\infty(\T^n)$, there exists a constant $C$ so that
	\begin{equation}
		\sup_{t\in [0,T]} \E \left| X^\eps_t - Z^\eps_t   \right|^2 \leq C \eps^2
		\label{eq:BrownianEst-strong}
	\end{equation}
	and
	\begin{equation}
		\sup_{t\in [0,T]} \left|\E\left[\varphi(X^\eps_t) - \varphi(Z^\eps_t)\right] \right|
		\leq C \eps^{2}T^2 \,.
		\label{eq:BrownianEst-weak}
	\end{equation}

\end{proposition}
Before we prove this result, we will gather a few facts. First, note that the solution to $\Vtilal_t$ is given by:
\begin{equation}\label{Valt}
	\Vtilal_t = V_0 e^{-t/\eps}
	+ \sqrt{\frac{2}{\eps}} \int_0^t e^{-(t-s)/\eps} \, dW_s.
\end{equation}
This motivates us to define:
\begin{equation}\label{Pt}
	P_t = \int_0^t e^{-(t-s)/\eps} \, dW_s.
\end{equation}
Note that $P_t$ satisfies:
\begin{equation}
	\eps \left( dW_t - d P_t  \right) = P_t \, dt \,,
	\label{eq:diff-relation}
\end{equation}
We have the following results on $P_t$.
\begin{lemma}
	The following are true.
	\label{lem:approx-1}
	\begin{equation}
		\E P_s \cdot W_s=\eps n \left(1-e^{-t/\eps}\right).
		\label{eq:PW}
	\end{equation}
	\begin{equation}
		c \eps^{p/2} \leq \E |P_t |^p \leq C \eps^{p/2} \,
		\label{ine:P}
	\end{equation}
	for some $c, C >0$ independent of $p$.
\end{lemma}
\begin{proof}
	To see~\eqref{eq:PW}, compute the expectation of $P_t\cdot W_t$ using It\^o isometry:
	\begin{equation*}
		\E \left(P_t\cdot W_t\right) =\sum_{i=1}^n \E \paren{P_t^iW_t^i}= n\int_0^t e^{-(t-s)/\eps}\,ds = \eps n \left(1-e^{-t/\eps}\right).
	\end{equation*}

	To see~\eqref{ine:P}, we note that $P_t$ is a Guassian process which satisfies
	\begin{equation*}
		P_t \sim \sigma_t Z \,,
	\end{equation*}
	where
	\begin{equation*}
		Z \sim \mathcal{N} \left( 0,  \mathrm{Id}  \right) \,,
		\qquad \sigma_t^2 = \frac{\eps}{2} \left( 1 -e^{-2t/\eps}  \right) \,,
	\end{equation*}
	and $\mathrm{Id}$ is the $n\times n$ identity matrix.
	Since $| Z |$ obeys $\chi$ distribution, we have
	\begin{equation*}
		\E | P_t |^p = \sigma_t^p \E | Z |^p
		= \sigma_t^p 2^{p/2} \frac{ \Gamma\left( \frac12(n +p)  \right)}{\Gamma\left( \frac12 n \right)}
		= \eps^{p/2} \left(1 - e^{-2t/\eps}  \right)^{p/2} \frac{ \Gamma\left( \frac12(n +p)  \right)}{\Gamma\left( \frac12 n \right)}\,,
	\end{equation*}
	from which~\eqref{ine:P} follows.
\end{proof}

Let us now return to the proof of Proposition \ref{prop:BrownianEst}.
\begin{proof}[Proof of strong estimate in Proposition \ref{prop:BrownianEst}]
	The velocity $\Vtilal_t$ in equation~\eqref{eq:bm-approx} is given by \eqref{Valt}, which together with \eqref{Pt} and \eqref{eq:diff-relation} yields:
	\begin{equation}
		\Vtilal_t \, dt =V_0 e^{-t/\eps} \, dt +\sqrt{\frac{2}{\varepsilon}}P_t \, dt =V_0 e^{-t/\eps} \, dt +\sqrt{2\varepsilon}(dW_t-dP_t).
	\end{equation}
	Integrating this once, we find that $X^\eps_t$ in \eqref{eq:bm-approx} is given by:
	\begin{equation}
		X^\eps_t =
		\eps(1 - e^{-t/\eps}) V_0 +\sqrt{2\varepsilon}\paren{W_t-P_t}.
		\label{eq:pathwise-solution-Y}
	\end{equation}
	Subtract $Z^\eps_t=\sqrt{2\eps} W_t$ from the above, we have
	\begin{equation*}
		X^\eps_t - Z^\eps_t
		=
		\eps(1 - e^{-t/\eps}) V_0-\sqrt{2\varepsilon}P_t.
	\end{equation*}
	Using the independence of $V_0$ and $P_t$ and \eqref{ine:P},
	\begin{equation*}
		\E \abs{X^\eps_t-Z^\eps_t}^2\leq \varepsilon^2 \E \abs{V_0}^2+2\varepsilon\E\abs{P_t}^2\leq C\varepsilon^2
	\end{equation*}
	for some positive constant $C$.
	This gives the inequality~\eqref{eq:BrownianEst-strong}.
\end{proof}

\begin{proof}[Proof of weak estimate in Proposition \ref{prop:BrownianEst}]
	By Taylor's theorem, we have
	\begin{align*}
		\varphi(Z^\eps_t )
		= & \varphi(0) + \partial_i \varphi(0) Z^{\eps,i}_t + \frac{1}{2} \partial_{ij}^2 \varphi(0) Z^{\eps,i}_t Z^{\eps,j}_t+ \frac{1}{3!} \partial_{ijk}^3 \varphi (0) Z^{\eps,i}_t Z^{\eps,j}_t Z^{\eps,k}_t+R_Z \\
		= & \varphi(0) + \partial_i \varphi(0) \sqrt{2\eps} W^{i}_t
		+ \eps  \partial_{ij}^2 \varphi(0) W^i_t W^j_t +\frac{ (2\eps)^{3/2}}{3!} \partial_{ijk}^3 \varphi (0) W^{i}_t W^{j}_t W^{k}_t
		+ R_Z,                                                                                                                                                                                                       \\
	\end{align*}
	where
	\begin{equation*}
		\abs{R_Z}\leq C\abs{Z_t^\varepsilon}^4=4\varepsilon^2 C\abs{W_t}^4  \text{ for some } C>0,
	\end{equation*}
	and $\partial_i$ etc. refer to partial derivatives with the $i$-th coordinate and the summation convention for repeated indices is in effect.
	Taking the expectation in the above, we see that:
	\begin{equation}\label{phiZ1}
		\E \abs{\varphi(Z^\eps_t )- \eps  \partial_{ij}^2 \varphi(0) W^i_t W^j_t}\leq \mc{O}(\eps^2).
	\end{equation}
	Once again, by Taylor's theorem and \eqref{eq:pathwise-solution-Y},
	\begin{align*}
		\varphi(X^\eps_t )
		          & =\varphi(0) + \partial_i \varphi(0) X^{\eps,i}_t
		+ \frac{1}{2} \partial_{ij}^2 \varphi(0) X^{\eps,i}_t X^{\eps,j}_t+ \frac{1}{3!} \partial_{ijk}^3 \varphi (0) X^{\eps,i}_t X^{\eps,j}_t X^{\eps,k}_t
		+ R_X,                                                       \\
		\abs{R_X} & \leq C\abs{X_t^\eps}^4 \text{ for some } C>0.
	\end{align*}
	Using \eqref{eq:pathwise-solution-Y}, given that $V_0$ is normally distributed, we have:
	\begin{equation}
		\E X^{\eps,i}_t=\E \paren{X^{\eps,i}_t X^{\eps,j}_t X^{\eps,k}_t}=0.
	\end{equation}
	Note that
	\begin{equation}
		\begin{split}
			X^{\eps,i}_t X^{\eps,j}_t= & \eps^2(1 - e^{-t/\eps})^2 V_0^iV_0^j +\sqrt{2\varepsilon}\varepsilon(1-e^{-t/\eps})V_0^i \paren{W_t^j-P_t^j}         \\
			                           & +\sqrt{2\varepsilon}\varepsilon(1-e^{-t/\eps})V_0^j \paren{W_t^i-P_t^i}+2\eps\paren{W_t^i-P_t^i}\paren{W_t^j-P_t^j}.
		\end{split}
	\end{equation}
	Therefore, we see from identities~\eqref{ine:P} and~\eqref{eq:PW} in Lemma \ref{lem:approx-1} that:
	\begin{equation}
		\left| \E (X^{\eps,i}_t X^{\eps,j}_t-2\varepsilon W_t^iW_t^j) \right|= C\varepsilon^2.
	\end{equation}
	Finally, using~\eqref{ine:P}, we have
	\begin{equation}
		\E \abs{X_t^\eps}^4 \leq C \E \eps^4|V_0|^4 + C \eps^2 \E (|W_t|^4 + |P_t|^4)
		\leq C\eps^2 t^2
	\end{equation}
	for some $C>0$.
	Combining the above, we have:
	\begin{equation}
		\left| \E\paren{\varphi(X^\eps_t )-\eps  \partial_{ij}^2 \varphi(0) W^i_t W^j_t}\right| \leq C\eps^2 t^2.
	\end{equation}
	Combining \eqref{phiZ1} with the above, we obtain estimate \eqref{eq:BrownianEst-weak}.
\end{proof}

\subsection{ Heuristics and Strong Estimate }
\label{subsec:derivation}
First, note that if $(\Xal_t, \Val_t)$ is a solution of~\eqref{eq:Langevin}, then
\begin{equation}
	\Val_t = V_0 e^{-t/\eps} + \frac{1}{\eps} \int_0^t e^{-(t-s)/\eps} b(\Xal_s,s) \, ds
	+ \sqrt{\frac{2}{\eps}} \int_0^t e^{-(t-s)/\eps} \, dW_s \,,
	\label{eq:Vsol}
\end{equation}
and
\begin{align}
	\Xal_t
	 & = \Xal_0 + \int_0^t \Val_s \, ds  \nonumber                                                                       \\
	 & = \Xal_0 +
	\int_0^t \left(  e^{-s/\eps} V_0 + \frac{1}{\eps}\int_0^s  e^{-(s-r)/\eps} b(X^\eps_r, r) \, dr \right) ds \nonumber \\
	 & \qquad + \sqrt{\frac{2}{\eps}} \int_0^t\int_0^s e^{-(s-r)/\eps} \, dW_r  \, ds \,.
	\label{eq:Langevin-solution}
\end{align}

We may then rewrite $\Xal_t$ to satisfy
\begin{subequations}
	\label{eq:rewrite}
	\begin{align}
		d\Xal_t    & =  \left( A_t + \Vtilal_t  \right) dt\,,                       \\
		d\Vtilal_t & = -\frac{\Vtilal_t}{\eps} \, dt + \sqrt{\frac 2 \eps} dW_t \,,
	\end{align}
	where
	\begin{equation*}
		A_t \defeq  b_0 e^{-t/\eps} + \frac{1}{\eps} \int_0^t e^{-(t-s)/\eps} b(\Xal_s,s) \, ds \,,
	\end{equation*}
	and
	\begin{equation}
		X_0 = x \in \T^n, \qquad \text{ $\Vtilal_0$ satisfies~\eqref{H3}} \,.
	\end{equation}
\end{subequations}
Notice now that, by integration by parts,
one may see that if we replace $X_t$ by a smooth function $x(t)$,
\begin{align*}
	A_t & = b(x(t),t) - \int_0^t e^{-(t-s)/\eps} \left(\partial_t b(x(s),s) + D b(x(s),s)\, \dot x(s) \right)  \, ds \\
	    & \approx b(x(t),t) - \eps\left(\partial_t b(x(t),t) + D b(x(t),t)\, \dot x(t) \right) \,.
\end{align*}

On the other hand, heuristically, one may guess by computing the quadratic variation that
\begin{equation}
	\label{eq:Vtilal}
	\begin{aligned}
		\sqrt{\eps} \Vtilal_t \, dt
		 & =  \sqrt{\eps} \Vtilal_0 e^{-t/\eps} dt + \sqrt{2} \int_0^t e^{-(t-s)/\eps} \, dW_s dt \\
		 & \approx \sqrt{\eps} \, d W_t \,.
	\end{aligned}
\end{equation}

The main result of this work is to show that this approximation gives
the error rates similar to those in Proposition~\ref{prop:BrownianEst}.
The guiding principle of the proof of the weak error estimate is based on the identity~\eqref{eq:diff-relation},
which we recall here for convenience
\begin{equation*}
	\eps \left( dW_t - d P_t  \right) = P_t \, dt \,.
\end{equation*}
This identity says that for any adapted process $G_t$,
$\int_0^t G_s  P_s \,ds$
is almost like a martingale, with a slight error.
In a sense, our main goal is to find a sharp quantification of this error as $\eps \to 0$.
By studying this carefully (see Sections~\ref{sec:limiting} and~\ref{sec:strong}),
we can see that
\begin{equation*}
	\E \abs{A_t - F(\Xal_t,t)}^p \approx \cO(\eps^p) \,.
\end{equation*}

\subsection{Weak estimate}
\label{subsec:weak-est}
The weak error estimate~\eqref{eq:BrownianEst-upgrade} is more delicate.
We start out with a standard approach of studying weak convergence of diffusion processes
via the backward Kolmogorov equation~\cite{HighamKloedenIntroductionNumerical2021}:
\step
Let $\varphi$ be a smooth function and $u^\eps$ be the solution to the following equation
\begin{subequations}
	\label{eq:PDE}
	\begin{equation}
		\partial_t u^\varepsilon +  F(x,t) \cdot \grad_x u^\varepsilon + \eps \Delta_x u^\varepsilon = 0 \,,
	\end{equation}
	with terminal data
	\begin{equation}
		u^\eps(\cdot,T) = \varphi \,.
	\end{equation}
\end{subequations}

Note that because $b \in C^\infty(\T^n \times [0, \infty))$, it is true that $F \in C^\infty(\T^n \times [0,\infty))$.
By regularity theory for Fokker-Planck equation (see~\cite[Theorem 3.2.4]{stroockMultidimensionalDiffusionProcesses2006} for example),
we have for each $k\in \mathbb{N}$, there exist a constant $C=C_T>0$, such that
\begin{equation}
	\label{eq:unibound}
	\sup_{\eps\in [0,1]} \sup_{0\leq t \leq T}\norm{ u^\eps(\cdot,t) }_{C^k} \leq C \,.
\end{equation}

\step
It follows from the definition \eqref{eq:PDE} that
\begin{equation*}
	u^\eps (\Xal_T,T)  =  \varphi(\Xal_T)
\end{equation*}
and,
after applying It\^o's formula to $u(Z^\varepsilon_t,t)$,
\begin{equation*}
	u^\eps (x,0) = \E(\varphi(\Zal_T)) \,.
\end{equation*}

We recall that, since we assume $V^\varepsilon_0 = b_0 + \tilde{V}^\varepsilon_0$,
we have
\begin{equation*}
	dX^\varepsilon_t = V^\varepsilon_t\,dt
	= \left(A_t + \tilde{V}^\varepsilon_t\right)\,dt
\end{equation*}
Therefore,
\begin{align}
	 & u^\eps (\Xal_T, T) -  u^\eps (x,0)                                                                      \nonumber         \\
	 & = \int_0^T \partial_t u(X^\varepsilon_t,t)
	+ \sum_{i=1}^n\int_0^T \partial_{x^i} u^\eps(X^\varepsilon_t,t)\,dX^{\varepsilon,i}_t                        \nonumber       \\
	 & =
	\int_0^T \partial_t  u^\eps ( \Xal_t, t) + (A_t + \Vtilal_t) \cdot \grad_x  u^\eps (\Xal_t, t) \, dt        \nonumber        \\
	 & =  \int_0^T \left( \Vtilal_t \cdot \grad_x  u^\eps (\Xal_t,t) - \eps \Delta_x  u^\eps (\Xal_t, t) \right) \, dt \nonumber \\
	 & \qquad + \int_0^T \left( A_t - F(\Xal_t,t)   \right) \cdot \grad_x  u^\eps(\Xal_t,t)\, dt  \nonumber                      \\
	 & = I + II
	\label{eq:I-II}
\end{align}
where we use that $\partial_t u^\eps = -F \cdot \nabla_x u^\eps - \varepsilon \Delta_x u^\eps$
in the third equality.

As before, the $A-F$ term should be small. The averaging effect  upgrades this difference to $\cO(\eps^2)$.
On the other hand, a notable structure that arises from our analysis is that
\begin{equation}
	\E  \int_0^T \left( \tilde V^\eps_t \cdot \grad_x  u^\eps (\Xal_t,t) - \eps \Delta_x  u^\eps (\Xal_t, t)\right)  dt  = \cO(\eps^2)\,.
	\label{ine:weak-I}
\end{equation}

The main difficulty, is to quantify the difference between
$\sqrt{2/\eps}\int_0^t Q(\Xal_s,s) P_s \, ds$
and the martingale $\sqrt{2\eps} \int_0^t Q(\Xal_s,s) dW_s$
for some function $Q:\T^n \times [0,\infty) \to \R$ with sufficient regularity.
Specifically, we want to study
\begin{equation*}
	\E  \int_0^t Q(\Xal_s,s) P_s \, ds \,.
\end{equation*}
Inspired by the proof of Proposition~\ref{prop:BrownianEst}, we exploit the observation that
since $Q(\E \Xal_s,s)$ is deterministic and $\E P_t = 0$,
\begin{align*}
	 & \E  \int_0^t Q(\Xal_s,s) P_s \, ds
	=
	\E  \int_0^t \left(Q(\Xal_s,s)  - Q(\E \Xal_s, s) \right)P_s\, ds                        \\
	 & = \E  \int_0^t \left(\Xal_s  - \E \Xal_s \right)\cdot\grad_x Q (f(s),s)  P_s\, ds \,,
\end{align*}
where $f(s) = \lambda_s\Xal_s+ (1- \lambda_s)\E \Xal_s$ for some $ \lambda_s\in[0,1]$.
Obtaining estimates for quantities similar to the RHS above is the crux of our proof.

\section{Limiting Langevin Calculus}
\label{sec:limiting}

As discussed above,
for an adapted process $G_t$, $\int_0^t G_s \cdot P_s \, ds$ behaves like a martingale with a slight error.
Therefore, while doing integration against $P_t$, one may be able to think in terms of It\^o isometry.
We will quantitatively uncover what this means in this section.

In what follows, it is understood that the underlying probability space is $(\Omega, \cF, \P)$.

\begin{lemma}
	\label{lem:martingale-approx}
	Let $T>0$,  $G:\Omega \times  [0,\infty) \to \R^{n}$ and  $H\in \Omega \times [0,\infty) \mapsto \mathrm{Sym}_n(\R)$ be processes
	that $G(\cdot,t), Q(\cdot,t)$ are $\cF_t$-adapted.
	Suppose that  $G^i(\omega,t), H^{ij}(\omega, \cdot) \in C^1([0,\infty);\R) $ and there are positive constant $C_0$, $C_1$ and $C_2$ such that
	$$ | H^{ij} | \leq C_0 \,, $$
	$$\E (H^{ij}_0)^2 + \E(G_0^i)^2 \leq C_1^2 \,,$$
	and
	$$\displaystyle \E  \int_0^T \abs{H^{ij\prime}_t}^2dt  + \E  \int_0^T \abs{G^{i\prime}_t}^2dt   \leq C_2^2$$ uniformly for $i,j$. Let
	$$J_t = \int_0^t e^{-(t-s)/\eps} H_s  dW_s \,.$$
	Then
	\begin{align}
		\left| \E  \int_0^T G(\cdot,  t)
		\cdot \sqrt{\frac{2}{\eps}} J_t \, dt \right|
		\leq C(1 + \sqrt{T})\eps
		\label{eq:EG-estimate}
	\end{align}
	and, consequently,
	\begin{align}
		\left| \E  \int_0^T G(\cdot,  t)
		\cdot \Vtilal_t \, dt \right|
		\leq C(1 + \sqrt{T})\eps
		\label{eq:EGV-estimate}
	\end{align}
	for some constant $C$ that depends only on $C_1,C_2$.
\end{lemma}
\begin{remark}
	When $H = \mathrm{Id}$, \eqref{eq:EG-estimate} becomes
	\begin{align}
		\left| \E  \int_0^T G(\cdot,  t)
		\cdot \sqrt{\frac{2}{\eps}} P_t  \, dt\right|
		\leq C(1 + \sqrt{T})\eps \,.
		\label{eq:EG-estimate-2}
	\end{align}
\end{remark}

\begin{lemma}
	\label{lem:diffusion-approx}
	Let $\mathrm{Sym}_n(\mathbb{R})$ be the set of $n\times n$ symmetric matrices.
	Let $T>0$ and $Q\in \Omega \times [0,\infty) \mapsto \mathrm{Sym}_n(\R)$ be a process that is $\cF_t$-adapted, where $Q^{ij}$ are its matrix elements.
	Suppose that  $Q^{ij}(\omega, \cdot) \in C^1([0,\infty);\R) $ and there are positive constant $C_1$ and $C_2$ such that
	$\E (Q^{ij}_0)^2\leq C_1^2$ and $\displaystyle \E  \int_0^T \abs{Q^{ij\prime}_t}^2dt   \leq C_2^2$ uniformly for $i,j$.
	Then there exists a constant $C >0$ that depends only on $C_1,C_2$  such that
	\begin{align}
		\E  \int_0^T \sum_{i,j = 1}^n P^i_t Q^{ij}_t  P^j_t \, dt
		= \frac{\eps}2\E \int_0^T \sum_{i=1}^n Q^{ii}_t \, dt + \cR
		\label{eq:diffusion-approx}
	\end{align}
	where
	\begin{equation}
		|\cR| \leq C(1 + \sqrt{T})\eps^{2}.
	\end{equation}
\end{lemma}

\begin{remark}
	\label{rem:ito}
	Lemmas~\ref{lem:martingale-approx} and~\ref{lem:diffusion-approx} give a justification (with explicit and optimal error rates) for the heuristics~\eqref{eq:Vtilal} because
	\begin{equation*}
		\E \int_0^T G \cdot \sqrt{\eps} dW_t = 0
	\end{equation*}
	and
	\begin{equation*}
		\E \int_0^T \sum_{i,j=1}^n Q^{ij} d \left[ \sqrt{\frac{\eps}{2}} W^i_t,\sqrt{\frac{\eps}{2}} W^j_t \right] = \frac{\eps}{2} \E \int_0^T \sum_{i=1}^n Q^{ii} \, dt \,.
	\end{equation*}
\end{remark}

\begin{proof}[Proof of Lemma~\ref{lem:martingale-approx}]
	To simplify our notations, we denote $G_t(\omega) = G(\omega,  t)$ and $H_t(\omega) = H(\omega,t)$.
	By definition of $J_t$, we have
	\begin{equation}
		\E |J_t|^2 = \E \int_0^t e^{-2(t-s)/\eps } \sum_{i=1}^n (H^{ii}_s)^2 \, ds
		\leq  \sum_{i=1}^n \int_0^t e^{-(t-s)/\eps} C_0^2 \, ds
		\leq C \eps \,.
		\label{ine:J-quadratic}
	\end{equation}
	\begin{equation}
		d J_t = -\frac{1}{\eps} J_t \, dt + H_t  dW_t \,.
		\label{eq:gen-diff-relation}
	\end{equation}

	Therefore,
	\begin{equation*}
		\int_0^T G_t \cdot \sqrt{\frac{2}{\eps}} J_t \, dt
		=
		\sqrt{2\eps} \int_0^T G_t \cdot \left( H_t dW_t - d J_t \right) \,.
	\end{equation*}

	Taking the expectation, we have
	\begin{equation}
		\begin{split}
			 & \E \int_0^T G_t \cdot \sqrt{\frac{2}{\eps}} J_t \, dt =- \sqrt{2\eps} \, \E \int_0^T G_t \cdot d J_t        \\
			 & = - \sqrt{2\eps} \, \E \left( G_T \cdot J_T \right) + \sqrt{2\eps} \, \E \int_0^T J_t \cdot G_t' \, dt  \,.
			\label{eq:Gt}
		\end{split}
	\end{equation}
	On the one hand, we have that, by H\"{o}lder inequality and~\eqref{ine:J-quadratic},
	\begin{equation}
		\left| \sqrt{2\eps} \, \E \left( G_T \cdot J_T \right)  \right| \leq
		\sqrt{2\eps} \left( \E |G_T|^2 \right)^{1/2} \left( \E  |J_T|^2 \right)^{1/2}\leq C \eps  \left( \E |G_T|^2 \right)^{1/2} \,.
		\label{ine:Gt}
	\end{equation}
	We also have
	\begin{equation}
		\begin{split}
			 & \paren{\E |G_T^2|}^{1/2} \leq \paren{\E |G_0|^2}^{1/2} + \paren{\E\abs{\int_0^T G'_t dt}^2}^{1/2}   \\
			 & \leq\left(\E G_0^2\right)^{1/2} +\sqrt{T}\paren{\E\int_0^T \abs{G'_t}^2\,dt}^{1/2}=C_1+\sqrt{T}C_2.
		\end{split}
	\end{equation}
	On the other hand, we have
	\begin{equation}
		\begin{split}
			 & \left|\sqrt{2\eps} \, \E \int_0^T J_t \cdot G_t' \, dt \right|
			\leq \sqrt{2\eps} \int_0^T \paren{\E\abs{J_t}^2}^{1/2}\paren{\E\abs{G_t'}^2}^{1/2}dt \\
			 & \leq\eps \sqrt{nT}\paren{\E\int_0^T \abs{G'_t}^2\,dt}^{1/2}
			\leq\eps \sqrt{nT}C_2.
			\label{ine:Gt2}
		\end{split}
	\end{equation}
	Plugging~\eqref{ine:Gt} and~\eqref{ine:Gt2} in~\eqref{eq:Gt}, we arrive at~\eqref{eq:EG-estimate}.

	To see~\eqref{eq:EGV-estimate}, we recall that $\Vtilal_t = \Vtilal_0 e^{-t/\eps} + \sqrt{ 2 / \eps } P_t$ and
	apply~\eqref{eq:EG-estimate-2} and H\"older inequality.
\end{proof}

\begin{proof}[Proof of Lemma~\ref{lem:diffusion-approx}]
	The proof of this lemma is similar to the previous one.
	\restartsteps
	\step
	Let us consider the case $i = j$.
	For notational simplicity, we write $Q^{ii}$ as $f$.
	By the It\^o formula,
	\begin{align*}
		d (P^i_t)^2=2 P^i_t dP^i_t +  dt = -\frac{2}{\eps} (P^i_t)^2 dt + 2P^i_t  dW^i_t + dt \,.
	\end{align*}
	where we used \eqref{eq:diff-relation}.
	Therefore,
	\begin{equation*}
		\int_0^T f_t (P^i_t)^2 \, dt =  - \frac{\eps}{2} \int_0^T   f_t d(P^i_t)^2 + \eps\int_0^T f_t P^i_t dW^i_t + \frac{\eps }{2} \int_0^T  f_t dt \,.
	\end{equation*}
	Taking expectation, we then have
	\begin{equation*}
		\E \int_0^T f_t (P^i_t)^2 \, dt =  - \frac{\eps}{2} \E \int_0^T  f_t d(P^i_t)^2 + \frac{\eps }{2} \E \int_0^T   f_t dt \,.
	\end{equation*}
	Now, define
	\begin{equation*}
		\cR_i = - \frac{\eps}{2} \E \int_0^T   f_t d(P^i_t)^2 \,
	\end{equation*}
	so that we write:
	\begin{equation*}
		\E \int_0^T f_t(P^i_t)^2\,dt = \frac{\eps }{2} \E \int_0^T   f_t dt + \cR_i.
	\end{equation*}
	The sum over $i$ of this is the diagonal part of \eqref{eq:diffusion-approx}.
	We now show that $|\cR_i|\lesssim \eps^2$.
	By integration by parts
	\begin{align*}
		\cR_i & = \frac{\eps}{2} \E \int_0^T   f_t d(P^i_t)^2 = \frac{\eps}{2}  \E \left( f_T (P^i_T)^2 \right) -  \frac{\eps}{2} \E \int_0^T (P^i_t)^2 f'_t \, dt
	\end{align*}
	In exactly the same way as in \eqref{ine:Gt2}, we have:
	\begin{equation}\label{eq:L2FT}
		\left(\E f^2_T\right)^{1/2} \leq C_1+\sqrt{T}C_2.
	\end{equation}
	Using \eqref{ine:P} from Lemma~\ref{lem:approx-1} and H\"older inequality
	on both terms on the right hand side
	then apply \eqref{eq:L2FT} in the first term, we have
	\begin{align*}
		\left|\E \left(f_T(P^i_T)^2\right)\right|
		                                             & \leq \left(\E f^2_T\right)^{1/2} \left(\E \left(P^i_T\right)^4\right)^{1/2}\leq \sqrt{\frac{3\eps^2}{4}}(C_1+\sqrt{T}C_2), \\
		\left| \E \int_0^T f'_t (P^i_t)^2\,dt\right| & \leq \int_0^T \left(\E (f'_t)^2\right)^{1/2} \left(\E(P^i_t)^4\right)^{1/2}\,dt \leq \sqrt{\frac{3\eps^2T}{4}} C_2.
	\end{align*}
	Combining the above estimates, we see that
	\begin{equation}\label{Rii}
		\abs{\cR_i}\leq C(1 + \sqrt{T})\varepsilon^2
	\end{equation}
	where the constant $C$ depends only on $C_1,C_2$.
	\step We now consider $i \not= j$.
	Again, for simplicity, we write $Q_t^{ij}$ as $f_t$.
	Using $P_t^i \stackrel{d}{=} P_t^j$ and $Q_t^{ij}=Q_t^{ji}=f$, we have
	\begin{align*}
		 & \E \int_0^T P^i_t  f_t P^j_t  \, dt = \eps \E \int_0^T P^i_t f_t \left( - dP^j_t + dW^j_t  \right) \, dt \\
		 & = - \eps \E \int_0^T P^i_t f_t dP^j_t
		= - \frac{\eps}2 \E \int_0^T \left( P^i_t f_t dP^j_t +  P^j_t f_t dP^i_t\right)                             \\
		 & = - \frac{\eps}2 \E \int_0^T f_t \left( P^i_t  dP^j_t +  P^j_t dP^i_t\right)
		= - \frac{\eps}2 \E \int_0^T f_t d(P^i_t P^j_t)                                                             \\
		 & = -\frac{\eps}2 \E \left( f_T P^i_T P^j_T - \int_0^T (P^i_t P^j_t) f'_t \, dt  \right)
	\end{align*}
	In exactly the same way as in the estimation of $R_i$ in Step 1 above,
	we obtain:
	\begin{equation}\label{Rij}
		\left| \int_0^T P^i_t f_t P^j_t \,dt \right| \leq C(1 + \sqrt{T}) \eps^2 \,.
	\end{equation}
	for some constant $C>0$ that depends only on $C_1, C_2$.
	\step
	Define
	\begin{equation*}
		\cR = \sum_{i=1}^n \cR_i + \sum_{i\not=j}^n \E \int_0^T P^i_t Q^{ij}_t P^j_t \, dt.
	\end{equation*}
	Combining \eqref{Rii} and \eqref{Rij}, we obtain \eqref{eq:diffusion-approx}
\end{proof}

To end this section, we note some simple a-priori estimates for $\Vtilal_t$ and $\Val_t$.

\begin{lemma}
	Assume that $\Vtilal_0$ satisfies assumption~\eqref{H3}.
	Then
	\begin{equation}
		\E \left| \Vtilal_t  \right|^p \leq C \,,
		\label{ine:Vtilde}
	\end{equation}
	and
	\begin{equation}
		\E \left|\int_0^t\Vtilal_s\, ds \right|^p \leq
		C  ( t^{p/2} \eps^{p/2} + \eps^p) \,,
		\label{eq:Vtilde-est}
	\end{equation}
	for some $C>0$, depending on $p$.
\end{lemma}

\begin{proof}
	We have that
	\begin{equation}
		\begin{aligned}
			\Vtilal_t
			 & = \Vtilal_0 e^{-t/\eps} + \sqrt{\frac{2}{\eps}} P_t \,.
		\end{aligned}
		\label{eq:Vtilda}
	\end{equation}
	Therefore, by~\eqref{ine:P} and that $\E |\Vtilal_0|^2 \leq ( \E | \Vtilal_0|^p )^{1/p}< \infty$,
	\begin{align*}
		\E \left| \Vtilal_t  \right|^p
		 & \leq \E \left( \left|\Vtilal_0\right|e^{-t/\eps}
		+ \sqrt{\frac{2}{\eps}} \left|P_t\right|\right)^p
		\leq 2^{p-1} \left(\E \left| \Vtilal_0 \right|^p e^{-pt/\eps}
		+ \left| \frac{2}{\eps} \right|^{p/2} \E |P_t|^p\right)             \\
		 & = 2^{p-1} \left( \E | \Vtilal_0 |^p e^{-pt/\eps} + C \right) \,,
	\end{align*}
	from which~\eqref{ine:Vtilde} follows.

	For \eqref{eq:Vtilde-est}, we write
	\begin{align}
		\int_0^t \Vtilal_s \, ds & =
		\Vtilal_0 \int_0^t e^{-s/\eps}\,ds +
		\sqrt{\frac{2}{\eps}} \int_0^t P_s\,ds \nonumber                                                               \\
		                         & = \eps \Vtilal_0 \left(1-e^{-t/\eps}\right) + \sqrt{2\eps} \left(-P_t + W_t\right).
		\label{eq:int-Vtilde}
	\end{align}
	Therefore,
	\begin{align*}
		\E
		\left| \int_0^t\Vtilal_s\,ds \right|^p
		 & \leq \E \left(\eps \left|\Vtilal_0\right|
		\left(1-e^{-t/\eps}\right) + \sqrt{2\eps} \left(
		|P_t| + |W_t|\right)\right)^p                                                                                            \\
		 & \leq 4^{p-1} \left(
		\eps^p \E \left|\Vtilal_0\right|^p \left(1-e^{-t/\eps}\right)^p
		+ (2\eps)^{p/2} \E|P_t|^p + (2\eps)^{p/2} \E|W_t|^p\right)                                                               \\
		 & \leq C \left(\eps^p \E \left|\Vtilal_0\right|^p \left(1-e^{-t/\eps}\right)^p + \eps^p\right) +C\eps^{p/2} t^{p/2} \,,
	\end{align*}
	from which~\eqref{eq:Vtilde-est} follows.
\end{proof}

\begin{lemma}
	Let $\Val_t$ be given by~\eqref{eq:Vsol}.
	Then there exists a constant $C>0$, depending only on $b$, such that
	\begin{equation}
		\E \left|\Val_t\right|^p \leq C \,.
		\label{eq:EV2}
	\end{equation}
\end{lemma}

\begin{proof}
	By triangle inequality, we have
	\begin{align*}
		\E \left|\Val_t\right|^p
		 & \leq \E \left(
		\left|\Val_0\right| e^{-t/\eps}
		+ \frac{1}{\eps} \int_0^t e^{-(t-s)/\eps} |b_s|\,ds
		+ \sqrt{\frac{2}{\eps}} |P_t|
		\right)^p                                             \\
		 & \leq C\E \left( \left|\Val_0\right|^p e^{-pt/\eps}
		+ \frac{1}{\eps^p} \left(\int_0^t e^{-(t-s)/\eps} |b_s|\,ds\right)^p
		+ \left( \frac{2}{\eps} \right)^{p/2}|P_t|^p
		\right)                                               \\
		 & \leq C \,.
	\end{align*}
	The last inequality follows from~\eqref{ine:P} and the fact that $\Val_0 = \Vtilal_0 + b_0$.
\end{proof}

\section{Proof of Strong Estimate~\eqref{ine:strong-est} }
\label{sec:strong}

As mentioned in subsection~\ref{subsec:derivation}, the key to the strong estimate
is to understand the difference between
$F(\Xal_t,t)$ and $A_t$.
This is the goal of this section.
To compress our notation, we denote
$$b_t = b(\Xal_t, t)\,, \qquad F_t = F(\Xal_t,t) \,.$$

\subsection{Drift approximations}

First, we estimate $L_t:=(\Vtilal_t + b_t)-\Val_t$.
From~\eqref{eq:Vsol} and integrating by parts, we have
\begin{equation}
	L_t := (\Vtilal_t + b_t) - \Val_t
	=   \int_0^t e^{-(t-s)/\eps} (\partial_t  b_s  + D_x  b_s \Val_s ) ds \,.
	\label{eq:V-diff-id}
\end{equation}
\begin{lemma}
	There exist  constants $C$ and $\eps_0$, depending only on $b$ and $p$, such that
	for $\eps \leq \eps_0$,
	\begin{equation}
		\E  \abs{L_T}^p  \leq  C \eps^p \,.
		\label{eq:V-diff}
	\end{equation}
\end{lemma}

\begin{proof}
	\restartsteps
	\step[First moment]
	By the H\"older inequality applied to \eqref{eq:V-diff-id} and by~\eqref{eq:EV2}.
	\begin{equation}
		\E \abs{ L_t  }  \leq C   \eps  \,.
		\label{eq:V-diff-L1}
	\end{equation}

	\step[Second moment]
	Next, note that
	\begin{equation*}
		\frac{dL_t}{dt} = -\frac{1}{\eps} L_t + (\partial_t b_t + D_x b_t \Val_t)
	\end{equation*}

	Therefore,
	\begin{equation*}
		\frac{d}{dt} \abs{ L_t }^2 = 2 L_t \cdot  \frac{d}{dt} L_t  = -\frac{2}{\eps} |L_t|^2 + 2 L_t \cdot ( \partial_t b_t + D_x b_t \Val_t ) \,.
	\end{equation*}

	Taking expectation of both sides, we have
	\begin{align}
		\frac{d}{dt} \E \abs{ L_t }^2
		 & = - \frac{2}{\eps} \E|L_t|^2 + 2 \E \left( L_t \cdot ( \partial_t b_t + D_x b_t \Val_t ) \right)                        \nonumber \\
		 & = -\frac{2}{\eps}\E  |L_t|^2 + 2\E \left( L_t \cdot ( \partial_t b_t)\right)
		+ 2 \E \left( L_t \cdot (  D_x b_t L_t ) \right) \nonumber                                                                           \\
		 & \qquad - 2\E ( L_t \cdot ( D_x b_t (\Vtilal_t + b_t) ) )  \nonumber                                                               \\
		 & \leq 2 \left(- \frac{1}{\eps} + \norm{b}_{C^1} \right) \E |L_t|^2
		+  2 \norm{b}_{C^1} \E |L_t|  - 2 \E ( L_t \cdot ( D_x b_t (\Vtilal_t + b_t) ) ) \nonumber                                           \\
		 & \leq 2 \left(- \frac{1}{\eps} + \norm{b}_{C^1} \right) \E |L_t|^2
		+ 2C\eps  - 2 \E ( L_t \cdot ( D_x b_t \Vtilal_t  ) ) \,,
		\label{eq:dL2}
	\end{align}
	where the last inequaltiy followed from~\eqref{eq:V-diff-L1}.
	Now, we have
	\begin{align}
		 & \abs{\E  \left( L_t \cdot (D_x b_t \Vtilal_t) \right)  }  \nonumber                                                                \\
		 & = \abs{ \E  \left(\int_0^t e^{-(t-s)/\eps} (\partial_t b_s + D_x b_s \Val_s) \, ds  \right) \cdot (D_x b_t \Vtilal_t)  } \nonumber \\
		 & \leq C  \int_0^t e^{-(t-s)/\eps} \norm{b}_{C^1}^2 \E ((1 +|\Val_s||\Vtilal_t|) \, ds  \nonumber                                    \\
		 & \leq C  \eps \,,
		\label{ine:dL2-last}
	\end{align}
	where the last inequality followed from~\eqref{ine:Vtilde} and~\eqref{eq:EV2}.
	Therefore,
	\begin{equation*}
		\frac{d}{dt} \E \abs{ L_t }^2 \leq 2 \left(- \frac{1}{\eps} + \norm{b}_{C^1} \right) \E |L_t|^2 + 2C \eps \,.
	\end{equation*}

	It follows that for $\eps\leq \eps_0 < 1/\norm{b}_{C^1}$,
	using integrating factor $ e^{(-1/\eps + \norm{b}_{C^1})2 t   } $ and the fact that $L_0 = 0$, we have
	\begin{align*}
		\E \abs{ L_T }^2
		 & \leq
		\frac{C\eps^2}{1 -  \eps \norm{b}_{C^1}} \,,
	\end{align*}
	from which~\eqref{eq:V-diff} follows.

	\step[$p$-th moment for $p\geq 3$]
	Suppose $\E |L_t |^{p-1} \leq C \eps^{p-1}$.
	Perform a similar computation as above we have
	\begin{align}
		\frac{d}{dt} \E \abs{ L_t }^p
		 & = - \frac{p}{\eps} \E|L_t|^p + p \E \left( |L_t|^{p-2} L_t \cdot ( \partial_t b_t + D_x b_t \Val_t ) \right)                        \nonumber \\
		 & = -\frac{p}{\eps}\E  |L_t|^p + p\E \left(|L_t|^{p-2} \left( L_t \cdot ( \partial_t b_t)\right) \right)
		+ p \E \left( |L_t|^{p-2} \left( L_t \cdot (  D_x b_t L_t ) \right)\right) \nonumber                                                             \\
		 & \qquad - p \E \left( |L_t|^{p-2} ( L_t \cdot ( D_x b_t (\Vtilal_t + b_t) ) ) \right)  \nonumber                                               \\
		 & \leq p \left(- \frac{1}{\eps} + \norm{b}_{C^1} \right) \E |L_t|^p
		+ pC\eps^{p-1}
		- p \E \left(|L_t|^{p-2} ( L_t \cdot ( D_x b_t \Vtilal_t  ) ) \right) \,.
		\label{eq:dLp}
	\end{align}

	We have
	\begin{align}
		 & \left| \E\left( |L_t|^{p-2} ( L_t \cdot ( D_x b_t \Vtilal_t  ) ) \right) \right|                                              \nonumber             \\
		 & \leq \left( \E |L_t|^{p-1} \right)^{\frac{p-2}{p-1}} \left( \E \left|  L_t \cdot D_x b_t \Vtilal_t \right|^{p-1} \right)^{\frac{1}{p-1}}  \nonumber \\
		 & \leq C \eps^{p-2}\left( \E \left|  L_t \cdot D_x b_t \Vtilal_t \right|^{p-1} \right)^{\frac{1}{p-1}} \,.
		\label{ine:dLp-1}
	\end{align}

	Now, by Minkowski inequality,
	\begin{align}
		 & \left( \E \left|  L_t \cdot D_x b_t \Vtilal_t \right|^{p-1} \right)^{\frac{1}{p-1}}                                                                 \nonumber     \\
		 & =  \left( \E \left| \int_0^t  e^{-(t-s)/\eps} (\partial_t b_s + D_x b_s \Val_s)  \cdot (D_x b_t \Vtilal_t) ds  \right|^{p-1}  \right)^{\frac{1}{p-1}} \nonumber   \\
		 & \leq    \int_0^t  \left(\E \left| e^{-(t-s)/\eps} (\partial_t b_s + D_x b_s \Val_s)  \cdot (D_x b_t \Vtilal_t) \right|^{p-1}\right)^{\frac{1}{p-1}} ds  \nonumber \\
		 & \leq \int_0^t e^{-(t-s)/\eps} \norm{b}_{C^1} \left(\E( 1 + |\Val_s| |\Vtilal_t|)^{p-1}\right)^{\frac{1}{p-1}} \, ds                  \nonumber                    \\
		 & \leq C \eps \,.
		\label{ine:dLp-2}
	\end{align}

	Applying estimates~\eqref{ine:dLp-1} and~\eqref{ine:dLp-2} to~\eqref{eq:dLp} we have
	\begin{equation*}
		\frac{d}{dt} \E | L_t |^p \leq p \left(  -\frac{1}{\eps} + \norm{b}_{C^1} \right)\E | L_t |^p + Cp \eps^{p-1}\,,
	\end{equation*}
	which implies
	\begin{equation*}
		\E | L_t |^p \leq C \eps^p \,.
	\end{equation*}

	By induction,~\eqref{eq:V-diff} is true.
\end{proof}

\begin{remark}
	One cannot apply Lemma~\ref{lem:martingale-approx} to derive~\eqref{ine:dL2-last} because $L_t^T D_x b_t$ does not satisfy the
	condition for $\int_0^T |G_t'|^2 \, dt $, as this quantity depends on $\eps$ in this case.
	The exponential structure of $L_t$ plays the saving role in this case.
\end{remark}

Next, we would like to estimate the difference $A_t - F_t$.
By definition,
\begin{align}
	A_t-F_t
	 & =b_0 e^{-t/\eps} + \frac{1}{\eps} \int_0^t e^{-(t-s)/\eps} b_s \, ds
	-  b_t + \eps \left( \partial_t b_t + D_x b_t b_t \right). \label{eq:A-F}
\end{align}

Using integration by parts, we compute the following integral
\begin{align}
	\frac{1}{\eps} \int_0^t e^{-(t-s)/\eps} b_s\,ds
	 & = b_t - b_0 e^{-t/\eps} - \int_0^t e^{-(t-s)/\eps} \left(\partial_t b_s + D_x b_s \Val_s\right) \, ds  \nonumber \\
	 & = b_t - b_0 e^{-t/\eps} - L_t \,.
	\label{eq:A-F-L}
\end{align}

Plugging~\eqref{eq:A-F-L} into~\eqref{eq:A-F}, we have
\begin{equation}
	A_t - F_t = -L_t + \eps (\partial_t b_t + D_x b_t b_t)
	\label{eq:A-F_3}
\end{equation}

Applying~\eqref{eq:V-diff} into this identity, we get the following proposition:

\begin{proposition}
	There exist constants $C, \eps_0>0$, depedning only on $b$, such that for
	$\eps < \eps_0$,
	\begin{equation}
		\E  |A_t - F_t|^p  \leq C \eps^p \,.
		\label{eq:A-F-est}
	\end{equation}
\end{proposition}

\subsection{Proof of strong estimate~\eqref{ine:strong-est}}
\label{subsec:strong}
We conclude this section with the proof of the strong estimate~\eqref{ine:strong-est}.
Let $(\Xal_t, \Val_t)$ and $\Zal_t$ be solutions of~\eqref{eq:Langevin}
and~\eqref{eq:advection-diffusion}, respectively, with
initial data $\Xal_0 = \Zal_0 =  z_0 \in \T^n$.
Note that
by identity~\eqref{eq:diff-relation},
\begin{align*}
	X_t^\varepsilon & = z_0 + \int_0^t A_s\,ds + \sqrt{2\varepsilon} \left(W_t - P_t\right), \\
\end{align*}
and by the definition,
\begin{equation*}
	Z_t^\varepsilon = z_0 + \int_0^t F_s\,ds + \sqrt{2\varepsilon}W_t.
\end{equation*}
We have that
\begin{align*}
	 & \E\left| \Xal_t - \Zal_t \right|^p                                                                 \\
	 & =
	\E\left|\int_0^t \left(A_s - F(\Zal_s,s)\right) \, ds
	- \sqrt{2\eps} P_t  \right|^p                                                                         \\
	 & \leq  C \int_0^t \E\left| A_s - F(\Xal_s,s)  \right|^p \, ds
	+ C\int_0^t \E\left| F(\Xal_s,s) - F(\Zal_s,s)  \right|^p \, ds
	+ C \eps^{p/2} \E\left|P_t   \right|^p                                                                \\
	 & \leq C \int_0^t \E\left| A_s - F(\Xal_s,s)  \right|^p \, ds
	+ C\norm{F}_{C^1}^p \int_0^t \E\left| \Xal_s - \Zal_s  \right|^p \, ds
	+ C \eps^p                                                                                            \\
	 & \leq  C\norm{F}_{C^1}^p \int_0^t \E\left| \Xal_s - \Zal_s  \right|^p \, ds  + C (1 + T) \eps^p \,.
\end{align*}
where we use \eqref{ine:P}  in the second-to-last inequality and~\eqref{eq:A-F-est} in the last inequality.

By Gronwall inequality, \eqref{ine:strong-est} holds. \qed

\section{Estimate $II$ in~\eqref{eq:I-II}}
\label{sec:A-F}

As mentioned in subsection~\ref{subsec:weak-est}, the averaging effect upgrades the
error rate of $A_t -F_t$ to $\cO(\eps^2)$.
We will study this in the current section.

First, we need to rewrite $A_t - F_t$.
Recall from~\eqref{eq:A-F_3} that
\begin{equation*}
	A_t - F_t = -L_t + \eps (\partial_t b_t + D_x b_t b_t)
\end{equation*}
where
\begin{equation*}
	L_t  =  \int_0^t e^{-(t-s)/\eps} (\partial_t  b_s  + D_x  b_s \Val_s ) ds \,.
\end{equation*}
Integrating $L_t$ by parts, we then have
\begin{equation}
	A_t - F(\Xal_t, t)  = \eps D_x b_t( b_t - \Val_t) + \hat R_t +  R_t + U_t \,,
	\label{eq:A-F_4}
\end{equation}
where
\begin{align}
	R_t       & = \eps \int_0^t e^{-(t-s)/\eps} D_x b_s dV^\eps_s    \label{eq:R}             \,,                                                  \\
	\hat{R}_t & = \eps e^{-t/\eps}(\partial_t b_0 + D_x b_0 \Val_0)\,,                                                                             \\
	U_t       & = \eps \int_0^t e^{-(t-s)/\eps} (\partial^2_t b_s + 2 \partial_t D_x b_s \Val_s + (D_x^2 b \Val_s)\Val_s ) \, ds\,. \label{eq:U_t}
\end{align}

\begin{remark}
	\label{rem:O2}
	From the above calculation, it is necessary for the term $\eps(\partial_t b_t + D_xb_t b_t)$
	to be present for us to achieve $\cO(\eps^2)$ as the $O(\eps)$ order coming from the intergration by parts
	of $L_t$ is controlled out by this term..
\end{remark}

We first note some easy bounds:

\begin{lemma}
	\label{lem:easy-est}
	There exists a constant $C>0$, depending only on $b$, such that
	\begin{equation}
		\E|\hat{R}_t|
		\leq C \eps e^{-t/\eps}  \,,
		\label{ine:C-hat}
	\end{equation}
	and
	\begin{equation}
		\E | U_t | \leq C \eps^2 \,.
		\label{ine:U}
	\end{equation}
\end{lemma}

\begin{proof}

	By H\"older inequality, we have
	\begin{equation*}
		\E|\hat{R}_t|
		\leq \eps e^{-t/\eps} \sqrt{n}\left\|b\right\|_{C^1}\left(1+ \E|V_0^\eps|\right) \,,
	\end{equation*}
	which leads to~\eqref{ine:C-hat}.

	From \eqref{eq:U_t}, we have
	\begin{align*}
		\E |U_t|
		 & \leq \eps \int_0^t e^{-(t-s)/\eps}
		\E \left(
		\left|\partial^2_t b_s\right|
		+ 2 \left|\partial_t D_x b_s \Val_s \right|
		+ \left| \left(D_x^2 b_s \Val_s\right)\Val_s\right|
		\right)\,ds                                                               \\
		 & \leq \eps C \left\|b\right\|_{C^2} \int_0^t e^{-(t-s)/\eps}
		\left(
		1 + 2\E \left|\Val_s\right| +\E \left|\Val_s\right|^2
		\right)\,ds                                                               \\
		 & \leq C \eps^2 \left\|b\right\|_{C^2} ( 1 + \left\| b \right\|_{C_0})^2
	\end{align*}
	Here, again, we use~\eqref{eq:EV2} and H\"older inequality for the last line.
\end{proof}

\begin{lemma}
	\label{lem:R}
	Let $\Phi \in C_b^\infty(\R^n\times [0,\infty); \R^n)$.
	There exists a constant $C>0$, depending only on $b$, such that
	\begin{equation}
		\abs{ \E \int_0^T R_t \cdot \Phi ( \Xal_t, t ) \, dt } \leq C(1+ T) \eps^2 \,.
		\label{ine:R}
	\end{equation}
\end{lemma}

\begin{proof}
	Let us write $\Phi_t = \Phi(\Xal_t,t)$.
	Recall from~\eqref{eq:V-diff-id} that $L_t =  (\Vtilal_t + b_t) - \Val_t $. Therefore,
	\begin{align*}
		d\Val_t
		 & = \frac{1}{\eps} \left(b_t-\Val_t\right)\,dt
		+ \sqrt{\frac{2}{\eps}}\,dW_t                                       \\
		 & = \frac{1}{\eps} \left(-\Vtilal_t + L_t\right)\,dt
		+ \sqrt{\frac{2}{\eps}}\,dW_t                                       \\
		 & = \frac{1}{\eps} \left(-\Vtilal_0 e^{-t/\eps} + L_t \right)\,dt
		+ \sqrt{\frac{2}{\eps}} \left(- \frac{1}{\eps}P_t\,dt + dW_t\right) \\
		 & = -\frac{1}{\eps} \left(\Vtilal_0 e^{-t/\eps} - L_t \right)\,dt
		+ \sqrt{\frac{2}{\eps}} dP_t \,,
	\end{align*}
	where the last equality follows from~\eqref{eq:diff-relation}.

	\begin{align*}
		R_t
		 & = \eps \int_0^t e^{-(t-s)/\eps} D_x b_s d\Val_s                                          \\
		 & =  - \int_0^t e^{-(t-s)/\eps} D_x b_s \left(  \Vtilal_0 e^{-s/\eps} - L_s  \right) \, ds
		+ \sqrt{2\eps} \int_0^t e^{-(t-s)/\eps} D_x b_s dP_s                                        \\
	\end{align*}

	We have
	\begin{align}
		 & \E \left|
		\int_0^t e^{-(t-s)/\varepsilon} D_x b_s \left(  \tilde{V}_0^\varepsilon e^{-s/\varepsilon}
		- L_s\right)\,ds \right|                                                              \nonumber        \\
		 & \leq \left\|b\right\|_{C^1}\int_0^t e^{-t/\varepsilon} \E \left|\tilde{V}^\varepsilon_0\right| \,ds
		+ \int_0^t e^{-(t-s)/\varepsilon}\E \left|D_x b_s L_s\right|\,ds                           \nonumber   \\
		 & \leq C  t e^{-t/\varepsilon}
		+ C \norm{b}_{C^1} \varepsilon^2 \,,
		\label{ine:R-1}
	\end{align}
	where we use \eqref{eq:V-diff} and the fact that $L_t = (\Vtilal_t + b_t) - \Val_t  $ in the last inequality.

	From~\eqref{ine:R-1}, it follows that
	\begin{align*}
		 & \abs{ \E \int_0^T R_t \cdot \Phi_t  \, dt }                                               \\
		 & \leq C \norm{\Phi}_\infty \int_0^T \left(  t e^{-t/\eps} + \eps^2  \right) \, dt
		+ C \sqrt{2\eps} \abs{ \E \int_0^T \Phi_t \cdot \int_0^te^{-(t-s)/\eps} D_x b_s dP_s \, dt } \\
		 & \leq C(1+T) \eps^2 \,,
	\end{align*}
	where the last inequality follows from~\eqref{eq:EG-estimate}
	with $H_s = D_x b_s$ and $G_t = \Phi_t$.
\end{proof}

Combining Lemmas~\ref{lem:easy-est} and~\ref{lem:R}, we arrive at the main result of this section:

\begin{proposition}
	Let $\Phi \in C_b^\infty(\R^n\times [0,\infty); \R^n)$.
	Then, there exists a constant $C>0$ such that
	\begin{equation}
		\abs{ \E \int_0^T ( A_t - F_t ) \cdot \Phi ( \Xal_t, t ) \, dt } \leq C(1+T) \eps^2 \,.
		\label{eq:weak-A-F}
	\end{equation}
\end{proposition}

\begin{proof}
	Let us write $\Phi_t = \Phi(\Xal_t,t)$.
	Using~\eqref{eq:A-F_4}, we have
	\begin{align*}
		 & \abs{ \E \int_0^T ( A_t - F_t ) \cdot \Phi_t \, dt }                                                                     \\
		 & = \abs{ \E \int_0^T \left( \eps D_x b_t (b_t - \Val_t ) + \hat R_t + R_t + U_t \right )  \cdot \Phi_t \, dt }            \\
		 & \leq \abs{ \E \int_0^T  \eps D_x b_t ( L_t - \Vtilal_t ) \cdot \Phi_t \, dt}
		+ \abs{\E \int_0^T\left( \hat R_t  + U_t \right )  \cdot \Phi_t \, dt }                                                     \\
		 & \qquad + \abs{\E \int_0^T  R_t    \cdot \Phi_t \, dt }                                                                   \\
		 & \leq \abs{ \E \int_0^T  \eps D_x b_t  L_t  \cdot \Phi_t \, dt}
		+ \abs{ \E \int_0^T  \eps D_x b_t  \Vtilal_t  \cdot \Phi_t \, dt}                                                           \\
		 & \qquad + \norm{\Phi}_\infty \E \int_0^T \left( \abs{\hat R_t}  + \abs{U_t}  \right)  \, dt
		+ \abs{\E \int_0^T  R_t    \cdot \Phi_t \, dt }                                                                             \\
		 & \leq C \norm{\Phi}_\infty \norm{b}_{C^1} \eps^2 + C \eps^2 + C \int_0^t (\eps e^{-t/\eps} + \eps^2) \, dt + C(1+T)\eps^2
	\end{align*}
	The first term comes from~\eqref{eq:V-diff}, the second from~\eqref{eq:EGV-estimate} and Lemma~\ref{lem:easy-est},
	and the last from~\eqref{ine:R}.
	Estimate~\eqref{eq:weak-A-F} follows immediately.
\end{proof}

Consequently, using $\Phi(x,t) = \grad u(x,t)$, we can then estimate the term $II$ from~\eqref{eq:I-II}:
\begin{corollary}
	Let $II$ be given in~\eqref{eq:I-II}. There exists a constant $C>0$ such that
	\begin{equation}
		\abs{\E(II) } \leq C(1+T) \eps^2 \,.
		\label{eq:E-II}
	\end{equation}
\end{corollary}

\section{Estimate I in~\eqref{eq:I-II}}
For convenience, let us recall
\begin{equation*}
	I =  \int_0^T \left( \Vtilal_t \cdot \grad_x  u^\eps (\Xal_t,t) - \eps \Delta  u^\eps (\Xal_t, t) \right) \, dt  \,.
\end{equation*}
The goal of this section is to establish~\eqref{ine:weak-I}.

\subsection{Auxiliary estimates}
\label{subsec:aux}

In order to proceed, we need a few auxiliary estimates for
the difference between $\Xal_t$ and its running average
whose proofs will be postponed to the Appendix~\ref{app:aux} to minimize
distractions from the main proof.

Define
\begin{equation}
	Y_t = \Xal_t - \E \Xal_t \,.
	\label{eq:Yt}
\end{equation}
\begin{lemma}
	\label{lem:Y4t}
	Then there exist $C_1, C_2 > 0$, depending only on $b$ such that
	\begin{equation}
		\E \left| Y_t  \right|^4 \leq C_1 (\eps^2 t^2 + \eps^4) e^{C_2 t} \,.
		\label{eq:Y4t}
	\end{equation}
\end{lemma}

\begin{lemma}
	\label{lem:running-ave-Y}
	Let $T>0$, $f\in C^1_b(\R)$ and $\Phi \in C_b^\infty(\R^n \times [0,\infty); \R^n)$.
	There exists a constant $C_T >0$ such that
	\begin{equation}
		\left| \int_0^T \E \left(P_t \cdot
		\int_0^t f(s) \left( \Phi( {\Xal_s}, s) - \Phi( \E {\Xal_s}, s )  \right) \, ds
		\right)\, dt
		\right|
		\leq C\left( 1 + \sqrt{T} \right)\eps^{5/2} \,.
		\label{eq:running-ave-Y}
	\end{equation}
\end{lemma}

\begin{lemma}
	\label{lem:average-approximate}
	There exist  constants $\varepsilon_0,C_1,  C_2 >0$, depending only on $b$, such that
	for $0<\varepsilon\leq \varepsilon_0$:
	\begin{equation}
		\sum_{j=1}^n \left| \E \left(\tilde V_0^{\eps,i} Y^j_t   \right)   \right|
		\leq  C_1\eps ( t  + 1)^2 e^{C_2 t}  \,.
		\label{eq:average-approximate}
	\end{equation}
\end{lemma}

As a consequence, we then have the following generalized estimate.
\begin{lemma}
	\label{lem:average-approximate-2}
	Let $G \in C^\infty_b(\R^n \times [0,\infty))$ and $(\Xal_t, \Vtilal_t)$ be solution of Equation~\eqref{eq:rewrite}.
	There exist $C_1, C_2 >0$, depending only on $b$, such that
	\begin{equation}
		\left| \E \left({\Vtilal[,i]_0} \left( G(\Xal_t, t) - G(\E \Xal_t, t)\right)  \right)   \right| \\
		\leq 		C_1 \eps(1+t)^2 e^{C_2 t} \,.
		\label{eq:average-approximate-upgrade}
	\end{equation}
\end{lemma}

Combining all the above lemmas, we have
\begin{proposition}
	\label{prop:weak-Vtil-est}
	Let $\Phi \in C_b^\infty(\R^n\times [0,\infty); \R^n)$.
	Then
	\begin{equation}
		\abs{ \E \int_0^T \Vtilal_t \cdot \int_0^t f(s)  \Phi( {\Xal_s}, s)   \, ds \, dt }
		\leq C \eps^2 \left(1 + \sqrt{T} \right) \,.
		\label{ine:weak-Vtil-est}
	\end{equation}
\end{proposition}

\begin{proof}
	We have that
	\begin{equation*}
		\Vtilal_t = \Vtilal_0 e^{-t/\eps} + \sqrt{\frac{2}{\eps}} P_t \,.
	\end{equation*}
	Therefore, as  $\Phi( \E {\Xal_s}, s )$ is deterministic,
	\begin{align*}
		 & \abs{ \E \int_0^T \Vtilal_t \cdot \int_0^t f(s)  \Phi( {\Xal_s}, s)  \, ds \, dt }                                                                                \\
		 & = \abs{ \E \int_0^T \Vtilal_t \cdot \int_0^t f(s) \left( \Phi( {\Xal_s}, s) - \Phi( \E {\Xal_s}, s )  \right) \, ds \, dt }                                       \\
		 & \leq
		\abs{ \int_0^T  e^{-t/\eps} \E  \left(\Vtilal_0 \cdot \int_0^t f(s) \left( \Phi( {\Xal_s}, s) - \Phi( \E {\Xal_s}, s )  \right) \, ds \right)\, dt }                 \\
		 & \qquad + \sqrt{\frac{2}{\eps}}\abs{ \int_0^T   \E  \left(P_t \cdot \int_0^t f(s) \left( \Phi( {\Xal_s}, s) - \Phi( \E {\Xal_s}, s )  \right) \, ds \right)\, dt } \\
		 & \leq \int_0^T \left( C_1 \eps(1+t)^2 e^{(C_2 - 1/\eps)t}\right) \, dt + C_3 (1+ \sqrt{T}) \eps^2                                                                  \\
		 & \leq C \eps^2 (1 + \sqrt{T}) \,.
	\end{align*}
	The first term in the second to last inequality comes from~\eqref{eq:average-approximate-upgrade} and the second term comes from~\eqref{eq:running-ave-Y}.
\end{proof}

Let us now resume to the main proof.

\subsection{Proof of~\eqref{ine:weak-I}}
\label{subsec:proof-I}

In this subsection we continue the proof of the weak estimate~\eqref{eq:BrownianEst-upgrade}.
The first two steps were performed in Subsection~\ref{subsec:weak-est}.

\restartsteps

\step
From~\eqref{eq:rewrite}, we have
\begin{align*}
	I
	 & = \int_0^T (\tilde{V}^\varepsilon_t\cdot \nabla_x u (X^\varepsilon_t,t) - \varepsilon \Delta_x u ( X^\varepsilon_t,t))\,dt                  \\
	 & =  \int_0^T  \Vtilal_t \cdot \grad_x  u(x, 0) \, dt                                                                                         \\
	 & \quad + \int_0^T  \Vtilal_t \cdot \int_0^t \left(\grad_x \partial_tu(\Xal_s,s) + D_x^2  u (\Xal_s, s) (A_s + \Vtilal_s)\right) \, ds  \, dt \\
	 & \quad - \int_0^T \eps \Delta  u (\Xal_t, t) \, dt
\end{align*}
We let
\begin{align*}
	IA & =
	\int_0^T \Vtilal_t \cdot \grad_x  u (x,0) \, dt,                                                                               \\
	IB & = \int_0^T \Vtilal_t \cdot \int_0^t  \left( \grad_x\partial_t  u(\Xal_s,s) + D_x^2 u (\Xal_s, s) A_s \right) \, ds \, dt, \\
	IC & = \int_0^T \Vtilal_t \cdot \int_0^t D^2_x  u (\Xal_s, s)  \Vtilal_s \, ds \, dt
\end{align*}
so that
\begin{equation}\label{eq:Is}
	I = IA + IB + IC - \eps \int_0^T  \lap_x  u(\Xal_t, t)\, dt.
\end{equation}

Immediately, we have
\begin{equation}
	\E (IA)  =
	\E \int_0^T \left( \Vtilal_0 e^{-t/\eps} + \sqrt{\frac{2}{\eps}} P_t \right)
	\cdot \grad_x u (x,0)\, dt
	= 0 \,.
	\label{eq:IA}
\end{equation}

We will be showing that
\begin{gather}
	\abs{ \E (IB)  }  \leq C \left( 1 + T \right)\eps^2 \,, \label{ine:IB} \\
	\abs{ \E \left( IC - \eps \int_0^T \Delta_x u(\Xal_t, t) \, dt   \right)  }  \leq C(1 + T) \eps^2 \label{ine:IC} \,.
\end{gather}

\begin{remark}
	Inequality~\eqref{ine:IB} is an upgrade of Lemma~\ref{lem:martingale-approx} as it takes into
	consideration the specific form of the function $G$.
	Inequality~\eqref{ine:IC}, up to technical details, is a consequence of the almost It\^o isometry
	in Lemma~\ref{lem:diffusion-approx}.
\end{remark}

{
\step[Analyzing $IB$]
We now show~\eqref{ine:IB} is true, .i.e.,
\begin{equation*}
	\abs{ \E (IB)  }  \leq C \left( 1 + T \right)\eps^2  \,.
\end{equation*}

Because
\begin{equation*}
	A_t  = b_0 e^{-t/\eps} + \frac{1}{\eps} \int_0^t e^{-(t-s)/\eps} b (\Xal_s, s) \, ds  \,,
\end{equation*}
lemmas in Subsection~\ref{subsec:aux} does not directly apply.
However, the strong estimate~\eqref{eq:A-F-est} comes to the rescue.

From~\eqref{eq:A-F_3} and definition of $F_t$ in~\eqref{eq:advection-diffusion}, we have
\begin{equation*}
	A_t  = F_t -L_t + \eps( \partial_t b_t + D_x b_t b_t) = b_t - L_t \,.
\end{equation*}
Then,
\begin{equation}
	IB
	= \int_0^T \Vtilal_t \cdot \int_0^t  \left( \grad_x\partial_t  u(\Xal_s,s) + D_x^2 u (\Xal_s, s) (b_s - L_s) \right) \, ds \, dt \,.
	\label{eq:IB-id}
\end{equation}
Let
\begin{equation*}
	\Phi(\Xal_t, t) = \grad_x\partial_t  u(\Xal_t,t) + D_x^2 u (\Xal_t, t) b_t \,.
\end{equation*}
From Proposition~\ref{prop:weak-Vtil-est}, we have
\begin{equation}
	\abs{\E \int_0^T \Vtilal_t \cdot \int_0^t  \Phi(\Xal_s,s )  \, ds \, dt}
	\leq C \eps^2(1 + \sqrt{T})
	\label{ine:IB-0}
\end{equation}

We are left to analyze
\begin{align*}
	 & \int_0^T \Vtilal_t \cdot \int_0^t  D_x^2 u(X_s,s) L_s \, ds \, dt                                                     \\
	 & =
	\int_0^T \left( \Vtilal_0 e^{-t/\eps} + \sqrt{\frac{2}{\eps}} P_t \right) \cdot \int_0^t  D_x^2 u(X_s,s) L_s \, ds \, dt \\
\end{align*}

Recall that $L_t = (\Vtilal_t + b_t) - \Val_t$.
By Fubini theorem,
\begin{align}
	 & \abs{ \E \int_0^T e^{-t/\eps} \Vtilal_0    \cdot \int_0^t  D_x^2 u(X_s,s) L_s \, ds \, dt } \nonumber                                       \\
	 & = \abs{ \E \int_0^T \int_s^T e^{-t/\eps} \Vtilal_0    \cdot   D_x^2 u(X_s,s) L_s \, dt \, ds } \nonumber                                    \\
	 & = \abs{ \int_0^T \eps \left( e^{-s/\eps} - e^{-T/\eps}  \right) \E \left(\Vtilal_0    \cdot   D_x^2 u(X_s,s) L_s \right)  \, ds } \nonumber \\
	 & \leq C\norm{u}_{C^2}   \int_0^T \eps  e^{-s/\eps} \left(\E   |L_s|^2 \right)^{1/2} \, ds  \nonumber                                         \\
	 & \leq C \int_0^T e^{-s/\eps} \eps^2 \, ds
	\leq C \eps^{3} \,.
	\label{ine:IB-1}
\end{align}
The last line follows from~\eqref{eq:V-diff}.

Furthermore, from~\eqref{eq:diff-relation},
\begin{align}
	 & \E \sqrt{\frac{2}{\eps}}\int_0^T   P_t  \cdot \int_0^t  D_x^2 u(\Xal_s,s) L_s \, ds \, dt \nonumber \\
	 & = \E \sqrt{2\eps} \int_0^T (dW_t - dP_t)\cdot \int_0^t  D_x^2 u(\Xal_s,s) L_s \, ds      \nonumber  \\
	 & = - \E \sqrt{2\eps} \int_0^T   \int_0^t  D_x^2 u(\Xal_s,s) L_s \, ds \cdot dP_t         \nonumber   \\
	 & =  \E \sqrt{2\eps} \left( - P_T \cdot \int_0^T D_x^2 u(\Xal_t,t) L_t \, dt
	+\int_0^T  P_t \cdot   D_x^2 u(\Xal_t,t) L_t \, dt                         \right)        \,.
	\label{eq:IB-2}
\end{align}
Recall that from~\eqref{ine:P} that $\E|P|^2 \leq C \eps$
and that $\E  | L_t |^2  \leq C\eps^2$ from~\eqref{eq:V-diff}.
Applying this to~\eqref{eq:IB-2}, it then follows that
\begin{equation}
	\abs{ \E \sqrt{\frac{2}{\eps}}\int_0^T   P_t  \cdot \int_0^t  D_x^2 u(\Xal_s,s) L_s \, ds \, dt}
	\leq C T \eps^2 \,.
	\label{ine:IB-2}
\end{equation}

Using~\eqref{ine:IB-0}, \eqref{ine:IB-1} and~\eqref{ine:IB-2} in~\eqref{eq:IB-id},
we have shown
\begin{equation*}
	\abs{\E (IB) } \leq C\left(1 + T  \right)\eps^2 \,.
\end{equation*}

\step[Analzing $IC$]
We now show~\eqref{ine:IC} is true, i.e.,
\begin{equation*}
	\abs{ \E \left( IC - \eps \int_0^T \Delta_x u(\Xal_t, t) \, dt   \right)  }  \leq C(1 + T) \eps^2  \,.
\end{equation*}

Denoting $R_t = D_x^2 u(\Xal_t, t)$, we have
\begin{align}
	IC  = & \int_0^T \Vtilal_t \cdot \int_0^t R_s  \Vtilal_s \, ds \, dt  \nonumber                                  \\
	      & = \int_0^T \Vtilal_t
	\cdot \int_0^t R_s  \left( \Vtilal_0 e^{-s/\eps} + \sqrt{\frac{2}{\eps}} P_s \right)  \, ds \Bigg)  dt \nonumber \\
	      & = IC1 + IC2 + IC3 + IC4 \,,
	\label{eq:IC}
\end{align}
where

\begin{gather*}
	IC1 = \int_0^T \Vtilal_0 e^{-t/\eps} \cdot \int_0^t R_s \Vtilal_0 e^{-s/\eps} \, ds \, dt \\
	IC2 = \int_0^T \Vtilal_0 e^{-t/\eps} \cdot \int_0^t R_s \sqrt{\frac{2}{\eps}} P_s \, ds \, dt \\
	IC3 = \sqrt{\frac{2}{\eps}} \int_0^T P_t \cdot \int_0^t R_s \Vtilal_0 e^{-s/\eps} \, ds \, dt \\
	IC4 = \frac{2}{\eps} \int_0^T P_t \cdot \int_0^t R_s P_s \, ds \, dt  \,.
\end{gather*}

\emph{Analyzing IC1.}
To analyze $IC1$, we note that by H\"{o}lder inequality,
\begin{align}
	\left| \E (IC1) \right|
	 & \leq \sum_{j=1}^n \sum_{k=1}^n \int_0^T \left(\E \left| \tilde{V}_t^{\varepsilon,j}\right|^2\right)^{1/2} e^{-t/\varepsilon}
	\int_0^t \left\|u\right\|_{C^2} \left(\E \left|\tilde{V}_s^{\varepsilon,k}\right|^2\right)^{1/2} e^{-s/\varepsilon}\,ds\,dt\nonumber \\
	 & \leq 2n^2 \int_0^T e^{-t/\varepsilon} \int_0^t e^{-s/\varepsilon}\,dt\,ds\nonumber                                                \\
	 & = 2n^2 \varepsilon^2 \left( \left(1-e^{-t/\varepsilon}\right)
	- \frac{1}{2} \left(1-e^{-2t/\varepsilon}\right)\right).
	\label{eq:IC1}
\end{align}
Here, we use \eqref{ine:Vtilde} in the second inequality.

\emph{Analyzing IC2.}
\begin{align}
	\left| \E (IC2 )   \right|
	 & = \left| \E \int_0^T \Vtilal_0 e^{-t/\eps} \cdot \int_0^t R_s\sqrt{\frac{2}{\eps}} P_s \, ds \, dt \right| \nonumber \\
	 & \leq \int_0^T e^{-t/\eps}  \int_0^t \left|  \E \left(
	(\Vtilal_0)^\top R_s\cdot \sqrt{\frac{2}{\eps}} P_s \, ds \right)\right| \, dt
	\nonumber                                                                                                               \\
	 & \leq C \eps^2 \,.
	\label{eq:IC2}
\end{align}
Here, the last inequality follows from Lemma~\ref{lem:martingale-approx},
with $G(\omega,t) = (\Vtilal_0)^\top R_t$.

\emph{Analyzing $IC3$.}
Applying relation~\eqref{eq:diff-relation} and integration by parts,
\begin{align*}
	IC3
	 & = \sqrt{\frac{2}{\eps}} \int_0^T P_t \cdot \int_0^t R_s \Vtilal_0 e^{-s/\eps} \, ds \, dt \\
	 & = \sqrt{2\eps}\int_0^T  \int_0^t R_s \Vtilal_0 e^{-s/\eps} \, ds \cdot dW_t               \\
	 & \qquad + \sqrt{2\eps} \left(
	-P_T \cdot \int_0^T R_t \tilde{V}_0^\varepsilon e^{-t/\varepsilon}\,dt
	+ \int_0^t P_t \cdot R_t \tilde{V}_0^\varepsilon e^{-t/\varepsilon}\,dt
	\right).
\end{align*}
Taking the expectation of the above
then applying H\"older inequality and~\eqref{ine:P}, we have
\begin{align}
	\E ( IC3 )
	 & \leq \sum_{j=1}^n \sum_{k=1}^n \sqrt{2\varepsilon} \left\|u\right\|_{C^2} \left(
	\left(\E \left|P^j_T\right|^2\right)^{1/2}
	\int_0^T \left(\E (\tilde{V}^{\varepsilon,k}_0)^2\right)^{1/2}
	e^{-t/\varepsilon}\,dt\right.\nonumber                                              \\
	 & \qquad \left.
	+ \int_0^T\left(\E \left|P^j_t\right|^2\right)^{1/2}
	\left(\E (\tilde{V}^{\varepsilon,k}_0)^2\right)^{1/2}
	e^{-t/\varepsilon}\,dt
	\right)\nonumber                                                                    \\
	 & \leq \varepsilon C \left\|u\right\|_{C^2} 2\int_0^T e^{-t/\varepsilon}\,dt
	\leq C\eps^2  \,.
	\label{eq:IC3}
\end{align}

\emph{Analyzing $IC4$.}
Utilizing~\eqref{eq:diff-relation} and integration by parts yet once again
\begin{align}
	IC4
	 & = 2 \left(- P_T \cdot \int_0^T R_t P_t\,dt
	+\int_0^T P_t \cdot R_t P_t\,dt\right) \nonumber       \\
	 & \qquad + 2\int_0^T \int_0^t R_s P_s\,ds \cdot dW_t.
	\label{eq:IC4-1}
\end{align}
After taking the expectation, the last term vanishes.
We shall estimate each remaining term separately.

Consider
\begin{align*}
	\left|\E \left(P^j_T \int_0^t R^{jk}_t P^k_t\right) \,dt\right|
	 & = \varepsilon \left|
	\E \left(-P^j_T \int_0^t R^{jk}_t\,dP^k_t\right)
	+ \E \left( P^j_T\int_0^t R^{jk}_t\,dW^k_t\right)
	\right|
\end{align*}
where we apply \eqref{eq:diff-relation}.
We have that, by It\^{o} isometry,
\begin{align}
	 & \varepsilon\left| \E \left( P^j_T \int_0^T R^{jk}_s \,dW^k_s \right) \right|
	= \varepsilon\left| \E \left( \delta^{jk} \int_0^T e^{-(T-s)/\eps} R^{jk}_s \, ds \right) \right|\nonumber \\
	 & \leq \varepsilon^2 \delta^{jk}\left\|u\right\|_{C^2} \left(1-e^{-T/\varepsilon}\right)
	\label{eq:IC4-2}
\end{align}
On the other hand, by H\"older inequality,
\begin{align}
	 & \varepsilon\left| \E \left( P^j_T \int_0^T R^{jk}_s dP^k_s \right)  \right|\nonumber    \\
	 & = \left| \eps \E \left( P^j_T \left(R^{jk}_T P^k_T
		- \int_0^T \sum_{\ell=1}^n P^k_t \partial_{x^\ell} R^{jk}_t V^{\varepsilon,\ell}_t
	+ \partial_t R^{jk}_t\right)\,dt \right) \right|\nonumber                                  \\
	 & \leq \frac{\varepsilon^2}{2} \|u\|_{C^2} \E |P^j_T|^2 \nonumber                         \\
	 & \qquad+ \varepsilon \left( \E \left| P^j_T\right|^4\right)^{1/4} \left\|u\right\|_{C^3}
	\int_0^T \left( \E \left|P^k_t\right|^4\right)^{1/4}
	\sum_{\ell=1}^n \left(\E \left|V^{\varepsilon,\ell}_t\right|^2\right)^{1/2}\,dt
	\nonumber                                                                                  \\
	 & \leq C \eps^2 T \,,
	\label{eq:IC4-3}
\end{align}
where
we apply \eqref{ine:P} and \eqref{eq:EV2} in the last inequality.

Lastly, the second term in \eqref{eq:IC4-1} is approximated by Lemma~\ref{lem:diffusion-approx},
\begin{align}
	 & \abs{ 2\E \int_0^T P_s \cdot R_s  P_s \, ds - \eps \E \int_0^T\lap_x u(\Xal_s,s) \,ds } \nonumber     \\
	 & = \abs{ 2\E \int_0^T P_s \cdot R_s  P_s \, ds -  \eps \E \int_0^T\sum_{j=1}^nR^{jj}_s \,ds} \nonumber \\
	 & \leq  C \left(1 + \sqrt{T} \right) \eps^2 \,.
	\label{eq:IC4-4}
\end{align}

Combining~\eqref{eq:IC4-1}-\eqref{eq:IC4-4}, we have
\begin{equation}
	\left|\E (IC4) - \eps \int_0^T  \lap_x  u(\Xal_t, t)\, dt\right| =
	C \left( 1 + T \right) \eps^2 \,.
	\label{eq:IC4}
\end{equation}

Combining~\eqref{eq:IC1},\eqref{eq:IC2}, \eqref{eq:IC3}, and~\eqref{eq:IC4},
we arrive at~\eqref{ine:IC}.

\step[Conclusion]
Combining~\eqref{eq:IA}, ~\eqref{ine:IB} and ~\eqref{ine:IC}, we have
\begin{equation}
	\left| \E (I)  \right|
	\leq C \left( 1 + T  \right)	\eps^2 \,,
	\label{ine:E-I}
\end{equation}
which is what~\eqref{ine:weak-I} says. \qed

\section{Proof of Weak Estimate~\eqref{eq:BrownianEst-upgrade}}
\label{sec:proof-weak}

We finish the proof of~\eqref{eq:BrownianEst-upgrade} here, which was started
in Subsection~\ref{subsec:weak-est}.

Combining~\eqref{ine:E-I} and~\eqref{eq:E-II} into ~\eqref{eq:I-II}, we then have
\begin{align*}
	 & \abs{ \E\left( \varphi(\Xal_T) - \varphi(\Zal_T)  \right) } = \abs{ \E \left( u^\eps(\Xal_T,T) - u^\eps(x,0)   \right)   } \\
	 & \leq \abs{ \E (I)} + \abs{ \E (II)} \leq C( 1 + T ) \eps^2 \,.
\end{align*}
The weak estimate~\eqref{eq:BrownianEst-upgrade} then follows. \qed

\subsection{Proof of Corollary~\ref{cor:rate}}
Note that
$ \norm{\cT b}_{L^\infty(\T^n \times [0,T]; \R^n) } \leq C <\infty$.
Estimate~\eqref{eq:Hasselmann-rate} is a consequence of this observation and
Gronwall inequality.

To see~\eqref{eq:weak-Hasselmann},
note that $\E (II) \leq C\eps^2$. However, $F = b - \eps\cT b$ and $\eps \cT b \sim \cO(\eps).$
So, without $\eps \cT b$, this estimate would be of order $\cO(\eps)$.
See Remark~\eqref{rem:O2} for further insights.
\qed

\section{Numerical simulation} \label{sec:numerics}
\label{sec:simulation}

In this section, we present two numerical investigations of our approximation regime.
The first part involves solving one-dimensional problems where
we verify the theory presented in the previous sections.
Since solving equation \eqref{eq:KFP} in one dimension is feasible,
we can compute the spatial density $g^\eps = \int \rho^\eps\,dv$, which will then be compared to $u^\eps$ solving \eqref{eq:PDEApprox}.
Likewise, solving \eqref{eq:Langevin} and \eqref{eq:advection-diffusion} is relatively simple.
The observed numerical convergence suggests that the estimates \eqref{ine:strong-est}
and \eqref{eq:BrownianEst-upgrade} are sharp,
and emphasize the importance of the drift correction term $\eps \mathcal{T}b$.
For the second part, we consider 2D problems.
Although quantifying the error in two dimensions is intractable for both \eqref{eq:KFP} and \eqref{eq:Langevin},
we can still observe the long-time behavior of the solutions
which exhibits non-uniformity even when the vector field $b$ is incompressible.
Our approximation captures this non-uniformity, suggesting its suitability as an approximation to long-time behavior.
A detailed analysis of this will be the subject of future work.

Before stating the numerical results, we provide a brief description of the numerical methods for \eqref{eq:KFP} and \eqref{eq:PDEApprox} in one dimensional system with periodic spatial boundary condition.
All of the computations were implemented in Julia~\cite{bezanson2017julia}.

\subsection{Numerical methods}

\subsubsection{Monte-Carlo simulation for SDEs}
We use SRIW1 solver~ \cite{RosslerRungeKutta2010} in Julia's package \texttt{DifferentialEquations.jl}~\cite{RackauckasNieDifferentialEquations2017} to simulate all the SDEs  presented in this work.

\subsubsection{Numerical method for~\eqref{eq:KFP} in 1D}

As mentioned in Section~\ref{sec:intro}, it is challenging to numerically study~\eqref{eq:KFP} for general dimension.
We propose a numerical method for Equation~\eqref{eq:KFP} in 1D ($x$ and $v$ are both 1D).
Let us rewrite the \eqref{eq:KFP} as
\begin{equation}
	\partial_t\rho = \mathcal{A}\rho + \mathcal{B}\rho,
\end{equation}
where $\mathcal{A}\rho:= - v\partial_x \rho$ and $\mathcal{B}(t)\rho := - \frac{1}{\eps}\partial_v((b(\cdot,t)-v)\rho-\partial_v\rho)$.
First, we consider the semi-discretization in time of $\rho$, denoted by $\rho^n\approx \rho(\cdot,t_n)$, where $t_n=n\delta t$,
which are defined recursively using Strang splitting method:
\begin{equation}
	\rho^{n+1} = e^{\mathcal{A}\delta t/2}e^{\mathcal{B}(t_n)\delta t}e^{\mathcal{A}\delta t/2}\rho^n, \quad \rho^0 = \rho(\cdot,0).
\end{equation}

Let us now denote by $\rho_{jk}^n$ the approximation of $\rho$ at $(x_j,v_k,t_n)$ where $x_j=j\delta x$, $v_k=k\delta v$.
We consider a truncated $v$-domain $(-M\delta v,M\delta v)$ for a large $M\in\mathbb{N}$, the boundary of which is assigned no-flux boundary condition so that the numerical scheme is conservative.

For the 1D simulation in Section~\ref{sec:num_conv_rate}, where we examine the rate of convergence, we choose $V=8$, $N=2^{11}$, $M=2^{12}$, and $\delta t = \sqrt{\eps} 2^{-10}$.

We use the convention that $\rho$ be approximated by $\rho^n_{jk}$ at half integer grid points, i.e., $j=\frac{1}{2},\dots,N-\frac{1}{2}$, $k=-M +\frac{1}{2},\dots,M - \frac{1}{2}$.
We define the following finite difference operators in $v$:
\begin{align}
	\left(D_v\rho^n_{j\cdot}\right)_{k} & := \frac{\rho_{j,k+ \frac{1}{2}}^n - \rho_{j,k - \frac{1}{2}}^n}{\delta v},\quad k=-M+1,\dots,M-1.
\end{align}
We also have averaging operators in $v$:
\begin{align}
	\left(A_v\rho^n_{j\cdot}\right)_{k} & := \frac{\rho_{j,k+ \frac{1}{2}}^n + \rho_{j,k - \frac{1}{2}}^n}{2},\quad k=-M+1,\dots,M-1.
\end{align}

We approximate the intermediate step $e^{\mathcal{B}\delta t}\rho(x_j,v_k,t_n)\approx \bar{\rho}_{jk}^n$ in the Strang splitting by applying Crank-Nicolson scheme:
\begin{subequations}
	\begin{equation}
		\frac{\bar{\rho}_{jk}^n - \rho^n_{jk}}{\delta t} =  - \frac{1}{\eps} \left(D_v q ^{n + \frac{1}{2}}_{j\cdot} \right)_{k}
	\end{equation}
	\begin{equation}
		q^{n+ \frac{1}{2}}_{jk} = \begin{cases}
			0,                                                                                                                                         & \text{ if }k=\pm M \\
			\left(\left(b(x_j,t_{n+ \frac{1}{2}}) - v_\cdot\right)A_v - D_v \right)\left(\frac{\bar{\rho}^{n}_{j\cdot}+\rho^{n}_{j\cdot}}{2}\right)_k, & \text{ if } -M<k<M \\
		\end{cases}
	\end{equation}
\end{subequations}
Effectively, the operator $e^{\mathcal{B}\delta t}$ is approximated by $B_+^{-1}B_-$ where $B_{\pm}$ is a $2M\times 2M$ matrix given by
\begin{equation}
	B_{\pm} = I \pm \frac{\delta t}{2\eps\delta v} \left(-D_+^\top \left(\mathrm{diag}(b-v) A - \frac{1}{\delta v}D_+\right)\right)
\end{equation}
where $D_+,A$ are $(2M-1)\times 2M$ matrices given by
\begin{equation}
	D_+ = \begin{pmatrix}
		-1     & 1      & 0      & \cdots & 0  & 0 \\
		0      & -1     & 1      & \cdots & 0  & 0 \\
		0      & 0      & -1     & \dots  & 0  & 0 \\
		\vdots & \vdots & \vdots & \ddots & 1  & 0 \\
		0      & 0      & 0      & \dots  & -1 & 1
	\end{pmatrix},\quad
\end{equation}
\begin{equation}
	A = \frac{1}{2}\begin{pmatrix}
		1      & 1      & 0      & \cdots & 0 & 0 \\
		0      & 1      & 1      & \cdots & 0 & 0 \\
		0      & 0      & 1      & \dots  & 0 & 0 \\
		\vdots & \vdots & \vdots & \ddots & 1 & 0 \\
		0      & 0      & 0      & \dots  & 1 & 1
	\end{pmatrix}.
\end{equation}

Next, notice that $e^{\mathcal{A}\delta t/2}$ is nothing but the translation operator in $x$ by $\frac{v\delta t}{2}$.
To obtain an accuracy of order $\mathcal{O}(\delta x^2)$, we use quadratic interpolation to approximate the translation.

\subsubsection{Numerical method for~\eqref{eq:PDEApprox} in 1D}
To solve \eqref{eq:PDEApprox} for $u$, naively, we could apply a Crank-Nicolson method in $x$.
However, this creates spurious oscillation for small $\eps$ where the equation becomes advection dominated.
To address this, we turn to an upwind method \cite{LeVeque2002},
as detailed below.

Note that the method describe here for 1D \eqref{eq:PDEApprox} is used for Section~\ref{sec:num_conv_rate}.
This can be extended to a 2D system by applying Strang splitting method where the operator is written as the sum of the 1D operators in both dimensions.

Let us denote its approximation by $u_j^n\approx u(x_j,t_n)$, $j=\frac{1}{2},\dots,N- \frac{1}{2}$.
For notational convenience, we have by convention that $u_0^n=u_N^n$, $u_{-1/2}^n = u_{N- 1/2}^n$ and $u_{-1}^n = u_{N-1}^n$.
We introduce backward and forward finite difference operators
\begin{align}
	\left(D^-_x u^n\right)_j & := \frac{u_j^n - u_{j-1}^n}{\delta x},   \\
	\left(D^+_x u^n\right)_j & := \frac{u_{ j+1 }^n - u_j^n}{\delta x},
\end{align}
for $j=0,\dots,N-1$.
We also denote by $A_x$ the average operator in $x$:
\begin{equation}
	\left(A_x u^n\right)_j := \frac{u_j^n + u_{j-1}^n}{2}\quad j=0,\dots,N-1.
\end{equation}
The approximation $u_j^n$ satisfies
\begin{subequations}
	\begin{equation}
		\frac{u^{n+1}_j-u^n_j}{\delta t} = - D_x^+ \left(\left(\left(b-\eps Tb\right) \left(x,t_{n+1/2}\right) A_x - \eps D^-_x\right)u^{n+1/2}\right)_j
	\end{equation}
\end{subequations}

We use $N=2^{19},\delta t = 2^{-7}$ for the setting in Section~\ref{sec:num_conv_rate},
and $N=2^9, \delta t = 2^{-7}$ for the 2D system in Section~\ref{sec:num_long_time}.

Let us consider an approximation $u_j^n$ given by the discretization scheme:
\begin{equation}
	\begin{split}
		\frac{u^{n+1}_j-u^n_j}{\delta t} & = - \frac{1}{\delta x}\left(a_{j + \frac{1}{2}} ^+ u^{n + \frac{1}{2}}_j - a_{j + \frac{1}{2}}^- u^{n + \frac{1}{2}}_{j+1}\right.                                                                                        \\
		                                 & \qquad \left.- a_{j - \frac{1}{2}} ^+ u^{n + \frac{1}{2}}_{j-1} + a_{j - \frac{1}{2}}^- u^{n + \frac{1}{2}}_j\right) + \frac{u_{j-1}^{n+ \frac{1}{2}}- 2 u_j^{n + \frac{1}{2}} + u_{j+1}^{n + \frac{1}{2}}}{\delta x^2}.
	\end{split}
\end{equation}
Here, we have $a^+ := \max \{0,a\}$ and $a^- := \max \{0,-a\}$ and $a_j = \left(b - \eps \cT b\right)(x_j,t_{n + \frac{1}{2}})$.
Since $a^+ = \frac{1}{2}(|a| + a)$ and $a^- = \frac{1}{2} (|a|-a)$, this is equivalent to
\begin{equation}
	\begin{split}
		\frac{u^{n+1}_j-u^n_j}{\delta t} & = - \frac{1}{\delta x}\left(a_{j + \frac{1}{2}}  \frac{u^{n + \frac{1}{2}}_j+u^{n + \frac{1}{2}}_{j+1}}{2} - \left|a_{j + \frac{1}{2}}\right|\frac{u^{n + \frac{1}{2}}_{ j+1 }-u^{n + \frac{1}{2}}_j}{2} \right. \\
		                                 & \qquad \left. - a_{j - \frac{1}{2}}  \frac{u^{n + \frac{1}{2}}_j+u^{n + \frac{1}{2}}_{j+1}}{2} + \left|a_{j - \frac{1}{2}}\right|\frac{u^{n + \frac{1}{2}}_{ j }-u^{n + \frac{1}{2}}_{ j-1 }}{2}\right)          \\
		                                 & \qquad + \frac{u_{j-1}^{n+ \frac{1}{2}}- 2 u_j^{n + \frac{1}{2}} + u_{j+1}^{n + \frac{1}{2}}}{\delta x^2}.
	\end{split}
\end{equation}
With this, we can write the implicit equation as a linear equation $B_+u^{n+1} = B_- u^n$ where
\begin{equation}
	B_{\pm} = I \pm \frac{\delta t}{2\eps\delta x} \left(-\mathrm{diag}\overline{a}\overline{A}+\mathrm{diag}|\overline{a}|\overline{A}+\mathrm{diag}\underline{a} \underline{A}-\mathrm{diag}|\underline{a}|\underline{A} + \frac{1}{\delta x}\overline{D}^\top \overline{D}\right).
\end{equation}
Here, upper and lower bi-diagonal matrices respectively with $\frac{1}{2}$ non-zero entries.
Here, the $N\times N$ matrices $\overline{A},\underline{A}$ and  $\overline{D}, \underline{D}$ are given by
\begin{equation}
	\overline{D} = \begin{pmatrix}
		-1     & 1      & 0      & \cdots & 0      & 0      \\
		0      & -1     & 1      & \dots  & 0      & 0      \\
		0      & 0      & -1     & \dots  & 0      & 0      \\
		\vdots & \vdots & \vdots & \ddots & \vdots & \vdots \\
		0      & 0      & 0      & \dots  & -1     & 1      \\
		1      & 0      & 0      & \dots  & 0      & -1
	\end{pmatrix} = -\underline{D}^\top,
\end{equation}
and $\overline{A} = |\overline{D}|/2, \underline{A}=\left|\underline{D}\right|/2$.
The $N$-vectors $\overline{a},\underline{a}$ are given by $\overline{a} = (a_1,\dots,a_N,a_0)$ and $\underline{a} = (a_0,a_1,\dots,a_N)$.

\subsection{Numerical results}

\subsubsection{Numerical Verification of Main Estimates} \label{sec:num_conv_rate}

We first consider the solutions $u^\eps$ and $\rho^\eps$ to \eqref{eq:PDEApprox} and \eqref{eq:KFP} in 1D
with the domain $\mathbb{T}=(0,2\pi)$ and $b(x,t) = \sin x \sin t$.
The initial conditions are
\begin{equation*}
	u^\eps(\cdot,0)\equiv 1/2\pi,\quad \rho^\eps(x,v,0) = (2\pi)^{-3/2}\exp(-v^2/2)
\end{equation*}
respectively.
In Figure \ref{fig:g_u_weak_cos_bsinxsint},
we observe the dependency on $\eps$ of the weak error $\left|\int_{\mathbb{T}} (u^\eps(\cdot,T) - g^\eps(\cdot,T))\varphi\right|$ when $T=1$,
where $g^\eps(x,t):=\int \rho^\eps (x,v,t)\,dv$ and $\varphi(x)=\cos kx$ for different modes $k=1,2,\dots,6$.
We indeed confirm the $\mathcal{O}(\eps^2)$ weak convergence rate as stated in the main theorem (Theorem~\ref{t:estimate}).
Here, we use the $\eps\in\{2^{-k}: k=0,1,2,\dots,6\}\cup\{2^{-k}:k=3.25,3.75,4.25,\dots,7.00\}$.

This result
is consistent with a Monte-Carlo simulation (Figure~\ref{fig:strong-weak}),
from which we can also confirm our strong convergence result of order $\cO(\eps)$.
Here, we approximate both the strong $\E \left|X^\eps_T-Z^\eps\right|$ the weak error $\E (\varphi(X^\eps_T)-\varphi(Z^\eps_T))$, $T=100$, by averaging over $500,000$ samples where $X^\eps$ and $Z^\eps$ solve \eqref{eq:Langevin} and \eqref{eq:advection-diffusion} respectively.
The initial condition is
\begin{equation*}
	X^\eps_0 = Z^\eps_0 = 0.
\end{equation*}

On the other hand,
the naive approximation \eqref{eq:naive} yields convergence rate of $\cO(\eps)$ (Figure~\ref{fig:g_utilde_weak_cos_bsinxsint}).
This illustrates that incorporating the drift correction term $\eps \cT b$ improve the accuracy of our approximation scheme.

\begin{figure}
	\begin{center}
		\includegraphics[width=0.95\textwidth]{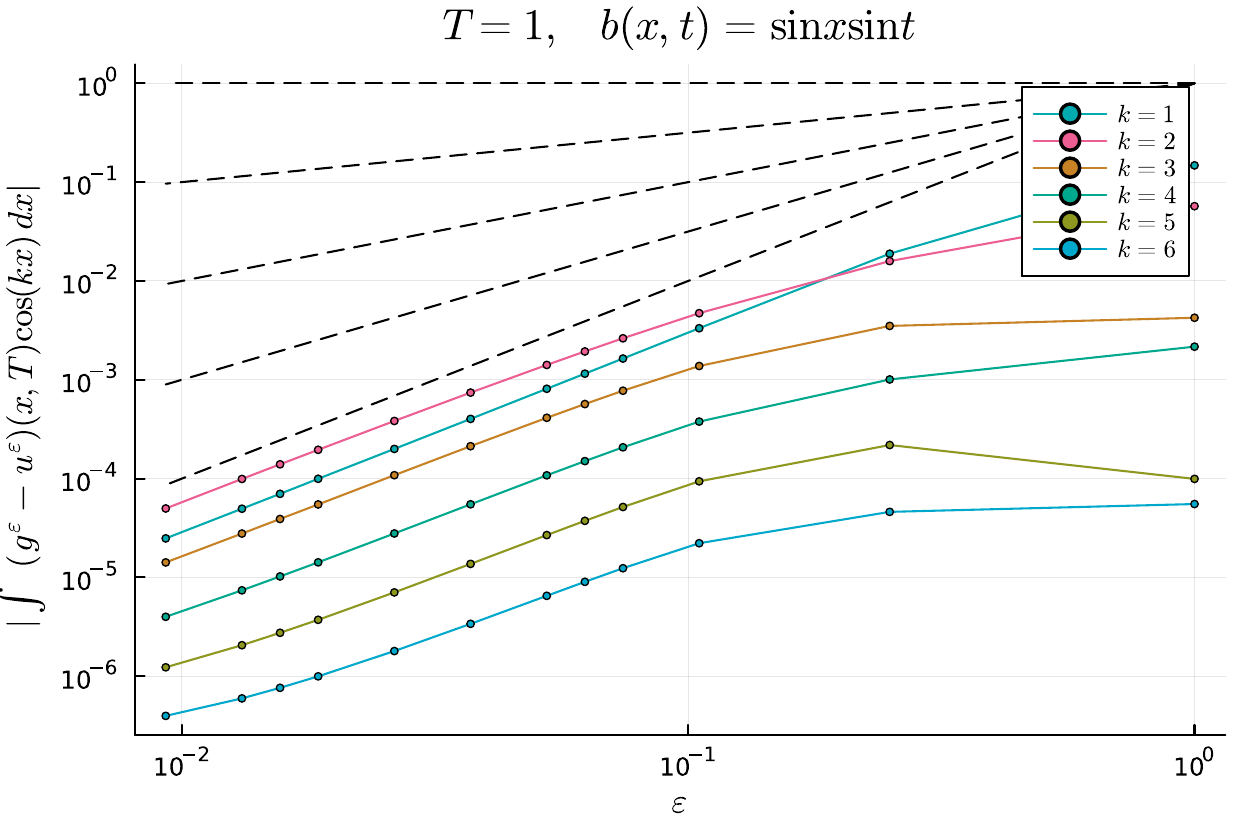}
	\end{center}
	\caption{
		The errors $\left| \int  (u^\eps(x,t)-g^\eps(x,t)) \cos(kx) \, dx \right|$ of the diffusion approximation \eqref{eq:PDEApprox} at time $T=1$ in $\log$ scale
		for Fourier modes $k=1,2,...,6$.
	}\label{fig:g_u_weak_cos_bsinxsint}
\end{figure}

\begin{figure}
	\begin{center}
		\includegraphics[width=0.95\textwidth]{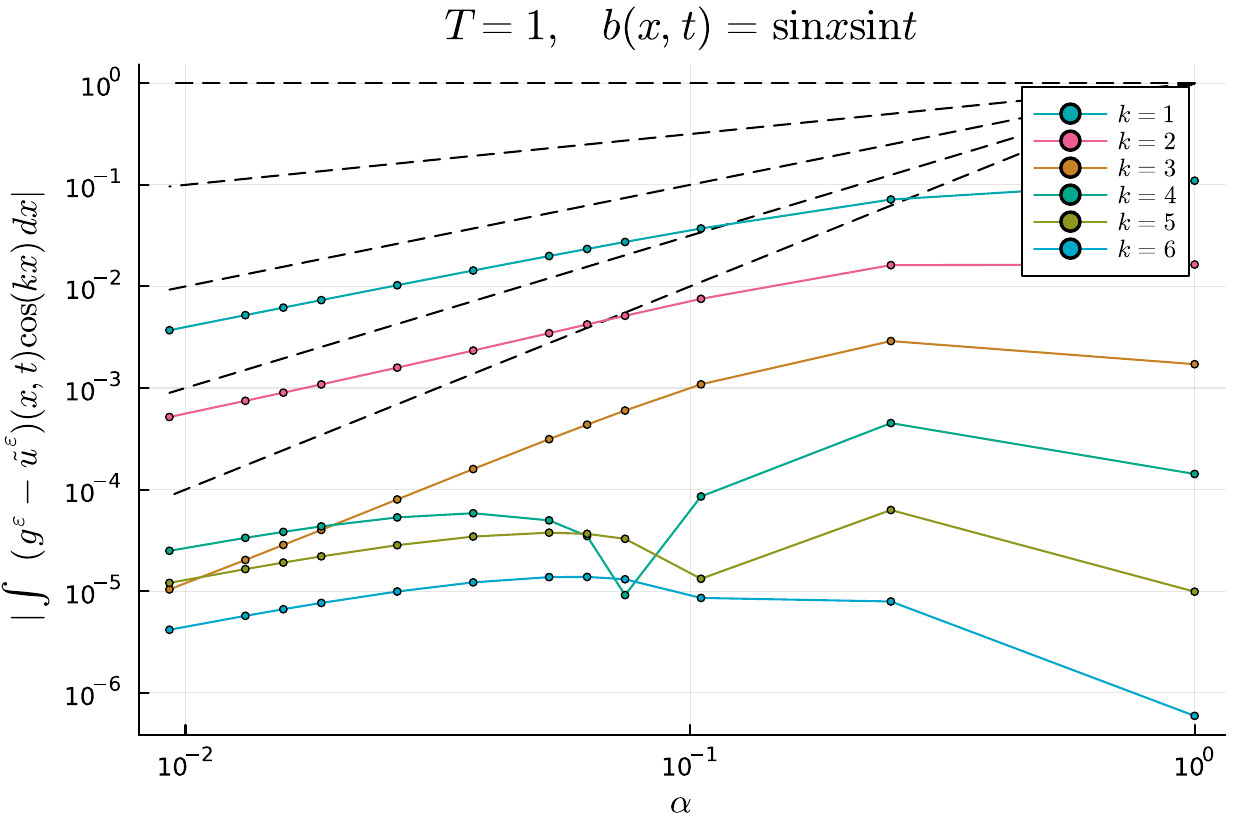}
	\end{center}
	\caption{
		The errors $\left| \int  (\tilde{u}^\eps(x,t)-g^\eps(x,t)) \cos(kx) \, dx \right|$ of the naive approximation \eqref{eq:naive} at time $T=1$ in $\log$ scale
		for Fourier modes $k=1,2,...,6$.
	}
	\label{fig:g_utilde_weak_cos_bsinxsint}
\end{figure}

\begin{figure}
	\begin{center}
		\includegraphics[width=0.95\textwidth]{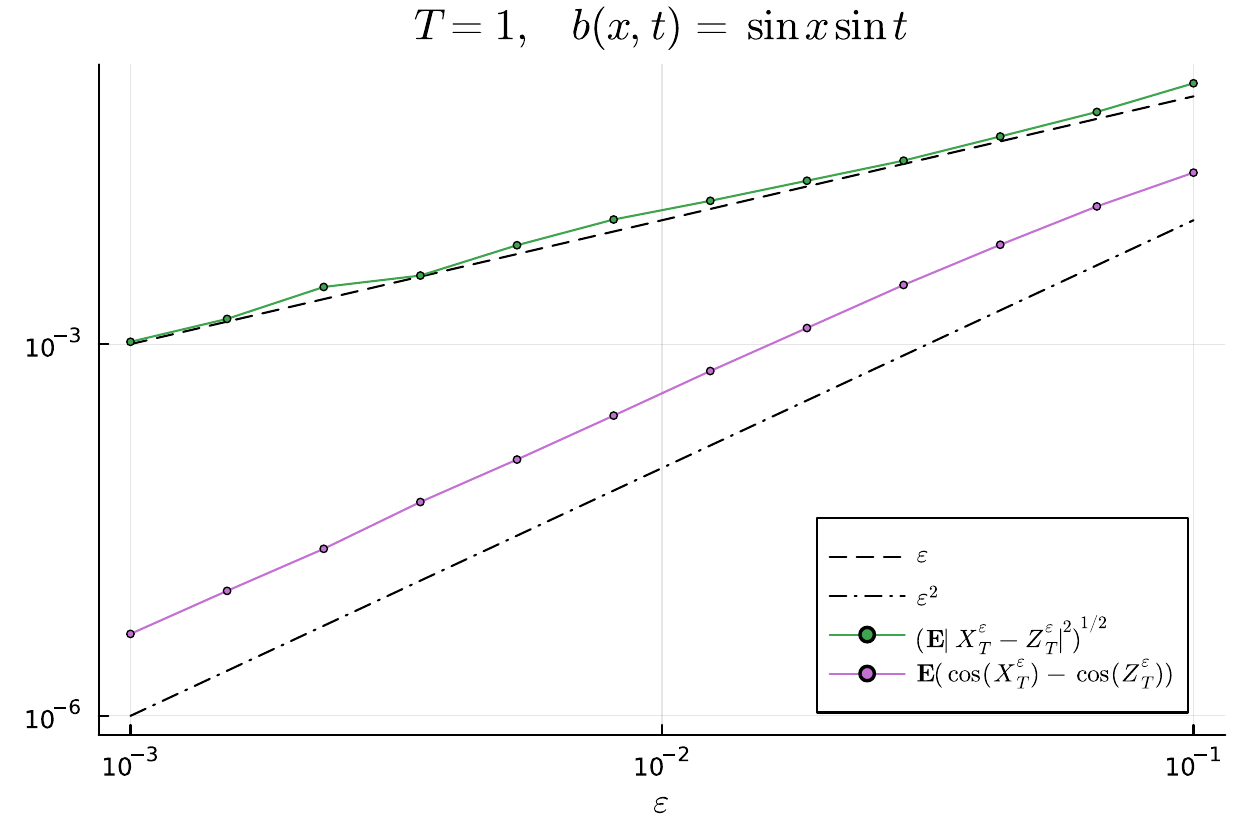}
	\end{center}
	\caption{Strong and weak errors in $\log$ scale via Monte-Carlo simulations between solution $X^\eps_t$ of~\eqref{eq:Langevin}
		and $Z^\eps_t$ of~\eqref{eq:advection-diffusion}
		at time $T=1$.
		500,000 trajectories were simulated using SRIW1 method, tested against $\varphi=\cos x$ for the weak error.}\label{fig:strong-weak}
\end{figure}

\subsubsection{Long time behavior} \label{sec:num_long_time}

For 2D simulations, we numerically solve the original problem only by a Monte-Carlo simulation of the Langevin SDEs \eqref{eq:Langevin} but not the kinetic Fokker-Planck equation \eqref{eq:KFP}.
We compute our approximation for both \eqref{eq:PDEApprox} and \eqref{eq:advection-diffusion}.
Additionally, we also compare our approximation to Monte-Carlo simulations of \eqref{eq:Hasselman-simple}, \eqref{eq:bakhtin-kifer}, and \eqref{eq:LWX}.
For the problem setup, we consider a 2D time-independent divergence-free flow $b(x)=(\sin(x_2),\sin(x_1))$ with fixed $\eps=2^{-4}$.
The simulation  admits a non-uniform stationary distribution.
To capture the long time behavior, we simulate all the equations to time $T=100$.
We demonstrate that the approximation \eqref{eq:advection-diffusion} captures the long-time behavior of \eqref{eq:KFP}, whereas \eqref{eq:Hasselmann} (or equivalently \eqref{eq:Hasselman-simple}) and \eqref{eq:kifer}
do not (Figures~\ref{fig:prev_approx}).

Figure~\ref{fig:longtime-approx}(A) and (B) describe Monte-Carlo simulation
with 1 million samples of the Langevin equation~\eqref{eq:Langevin} and the approximation~\eqref{eq:advection-diffusion}, respectively.
Then, we reconstruct the densities $\rho^\eps$  of~\eqref{eq:Langevin} and $\tilde \rho^\eps$ of~\eqref{eq:advection-diffusion}
by kernel density estimation in Julia's library KernelDensity.jl.
Figure \ref{fig:longtime-approx} (C) describes solution $u^\eps$ of the PDE~\eqref{eq:PDEApprox}.
We see that the long-time behavior of the approximation captures that of the original equation very well.
The $L^2$ errors are $\norm{\rho^\eps - u^\eps}_{L^2} \approx 0.0009$
and $\norm{\rho^\eps - \tilde \rho^\eps}_{L^2} \approx 0.0007$.
The $L^\infty$ errors are $\norm{\rho^\eps - u^\eps}_{L^\infty} \approx 0.007 $
and $\norm{\rho^\eps - \tilde \rho^\eps}_{L^\infty} \approx 0.003$.

Note that without the higher order correction in the advection term the approximated steady state solution would be merely the constant solution.
In fact, this is not surprising since one see that the drift term of their approximations are divergence-free vector fields.

\appendix

\section{Proofs of Lemmas }
\label{app:aux}

\begin{proof}[Proof of Lemma~\ref{lem:Y4t}]
	Explicitly, we have
	\begin{align*}
		Y_t & = \int_0^t \Big( \Val_0 e^{-s/\eps} + \frac{1}{\eps} \int_0^s e^{-(s-r)/\eps} (b_r - \E b_r) \, dr
		+ \sqrt{\frac{2}{\eps}} P_s  \Big) ds \,.
	\end{align*}
	Consider
	\begin{align*}
		 & \E \left| \int_0^t (b_s - \E b_s) \, ds \right|^4                                                   \\
		 & = \E \left|  \int_0^t \left(b_s - b(\E X_s,s) - \E (    b_s -  b(\E X_s,s)) \right) \, ds \right|^4 \\
		 & =
		\E \left|  \int_0^t \left( D_xb(\eta_s,s) ( \Xal_s - \E \Xal_s)
		- \E \left(D_xb(\eta_s,s) (\Xal_s - \E \Xal_s  )   \right)\right) \, ds  \right|^4                     \\
		 & \leq
		\int_0^t \left( \E \left| D_xb(\eta_s,s) ( \Xal_s - \E \Xal_s)\right|^4
		+ \E \left| \left(D_xb(\eta_s ,s) (\Xal_s - \E \Xal_s  )   \right)\right|^4 \right)\, ds               \\
		 & \leq 2\norm{D_xb}^4_\infty \int_0^t \E \left| \Xal_s - \E \Xal_s    \right|^4 \, ds \,.
	\end{align*}

	Therefore,
	\begin{align*}
		\E \left| Y_t \right|^4
		 & \leq
		2 \norm{D_xb}_\infty \int_0^t \E\left| Y_s \right|^4 \, ds + 4 \eps^2 \E \left| W_t \right|^4 + C\eps^4{\E|\tilde{V}_0|^4} \\
		 & \leq
		2 \norm{D_xb}_\infty \int_0^t \E\left| Y_s \right|^4 \, ds + C\eps^2t^2 + C\eps^4 \,.
	\end{align*}
	By Gronwall's inequality, we have
	\begin{equation*}
		\E \left| Y_t  \right|^4 \leq C_1 (\eps^2 t^2 + \eps^4) e^{C_2 t} \,,
	\end{equation*}
	as desired.
\end{proof}

\begin{proof}[Proof of Lemma~\ref{lem:running-ave-Y}]
	Let $\Theta_t = \Phi({\Xal_t},t)  - \Phi(\E \Xal_t, t)$.
	By integrating by parts,
	\begin{align*}
		 & \E \left(P_T \cdot \int_0^T f(t) \Theta_t \, dt\right) \\
		 & =
		\E \left( \int_0^T \int_0^t f(s) \Theta_s \, ds \cdot dP_t   + \int_0^T f(t) \Theta_t \cdot P_t \, dt \right) \,.
	\end{align*}
	By the identity~\eqref{eq:diff-relation}, we have
	\begin{align*}
		 & \E \left( P_T \cdot \int_0^T f(t) \Theta_t \, dt   \right) \\
		 & =
		\E \left( \int_0^T \int_0^t f(s) \Theta_s \, ds \cdot \left( -\frac{P_t}{\eps} \, dt + dW_t \right)
		+ \int_0^T f(t) \Theta_t \cdot P_t \, dt \right)              \\
		 & =
		-\frac{1}{\eps} \int_0^T \E \left(  P_t  \cdot \int_0^t f(s) \Theta_s \, ds\right)  dt
		+ \E \int_0^T f(t) \Theta_t \cdot P_t \, dt   \,.
	\end{align*}
	Let $\displaystyle N_T = \int_0^T\E \left(  P_t \cdot  \int_0^t f(s) \Theta_s \, ds\right) dt $.
	We then have
	\begin{equation*}
		\frac{d}{dT} N_T = -\frac{1}{\eps}  N_T  + \E \int_0^T f(t) \Theta_t \cdot P_t \, dt \,.
	\end{equation*}
	Therefore,
	\begin{equation*}
		N_T = \int_0^T e^{-(T-t)/\eps} \E \int_0^t f(s) \Theta_s \cdot P_s \, ds \, dt \,.
	\end{equation*}
	By~\eqref{eq:EG-estimate}, we have
	\begin{equation*}
		\left| N_T \right| \leq C\left( 1 + \sqrt{T}  \right) \eps^{5/2}  \,,
	\end{equation*}
	from which~\eqref{eq:running-ave-Y} follows.
\end{proof}

\begin{proof}[Proof of Lemma~\ref{lem:average-approximate}]

	\restartsteps

	\step
	Now, let $M^j_t= b^j_t - b^j(\E \Xal_t,t)$ so that we write
	\begin{align}
		 & \E {\Vtilal[,i]_0} Y^j_t                                                                \nonumber \\
		 & =
		\E \int_0^t {\Vtilal[,i]_0} \left( {\Vtilal[,j]_0} e^{-s/\eps}
		+ \frac{1}{\eps} \int_0^s e^{-(s-r)/\eps} M^j_r \, dr \right)\,ds     \nonumber                      \\
		 & \qquad+ \E\int_0^t\Vtilal_0\left(-\frac{1}{\eps}\int_0^se^{-(s-r)/\eps}\E M^j_r\,dr
		+\sqrt{\frac{2}{\eps}} P^j_s  \right) ds\nonumber                                                    \\
		 & =
		\E \int_0^t {\Vtilal[,i]_0} \left( {\Vtilal[,j]_0} e^{-s/\eps}
		+ M^j_s - M^j_0e^{-s/\eps} - \int_0^s e^{-(s-r)/\eps}\, dM^j_r \right)\,ds    \nonumber              \\
		 & = \delta^{ij} \eps(1-e^{-t/\eps})
		+ \int_0^t \E (\Vtilal[,i]_0 M^j_s)\,ds
		- \int_0^t\int_0^s e^{-(s-r)/\eps}\, dM^j_r \,ds \label{ine:GWij}
	\end{align}
	Here, the last integral in the first equality vanishes
	and we apply integration by parts in the second equality.

	Now we estimate both integrals on the right-hand side.
	Using that
	\begin{equation}
		M^j_s = \nabla_x b^j(\E \Xal_s,s)\cdot Y_s + \frac{1}{2} Y_s \cdot D_x^2 b^j(\xi_s,s) Y_s\,ds,
		\quad 0\leq s\leq t,
		\label{eq:Mjs}
	\end{equation}
	for some intermediate value $\xi_s$ between $\Xal_s$ and $\E\Xal_s$,
	we have
	\begin{align*}
		\left|\int_0^t \E (\tilde{V}_0^{\varepsilon,i} M^j_s)\,ds\right|
		 & \leq \left\|b\right\|_{C^1} \int_0^t \left(
		\left|\E \sum_k \tilde{V}_0^{\varepsilon,i} Y^k_s\right|
		+ \frac{1}{2} \E \left(\left|Y_s\right|^2\left| \tilde{V}_0^{\varepsilon,i}\right|\right)
		\right)\, ds                                   \\
		 & \leq \left\|b\right\|_{C^1} \int_0^t \left(
		\sum_k \left|\E \tilde{V}_0^{\varepsilon,i} Y^k_s\right|
		+ \frac{1}{2} \left(\E \left|Y_s\right|^4\right)^{1/2}
		\right)\, ds                                   \\
		 & \leq \left\|b\right\|_{C^1} \int_0^t
		\sum_k \left|\E \tilde{V}_0^{\varepsilon,i} Y^k_s\right|\,ds
		+ C_1(\eps t^2 + \eps^2t) e^{C_2 t}                   \,.
	\end{align*}
	Here, we use \eqref{eq:Y4t} in the third inequality.

	Similarly, using \eqref{eq:Mjs} and write $\Xi_r = \nabla b_r^jV_r + \partial_t b^j_r$, we have
	\begin{align*}
		 & \left|\E \left(\tilde{V}_0^{\varepsilon,i}\int_0^t\int_0^s e^{-(s-r)/\eps}\,dM^j_r\,ds\right)\right| \\
		 & = \left|\E \left(\tilde{V}_0^{\varepsilon,i}\int_0^t\int_0^s
		e^{-(s-r)/\varepsilon} \left(\Xi_r - \E\Xi_r\right)\,dr\,ds\right)\right|                               \\
		 & \leq \left(\E \left(\tilde{V}_0^{\varepsilon,i}\right)^2\right)^{1/2}
		\left(\int_0^t \int_r^t e^{-(s-r)/\varepsilon}
		\left(\E \left|\Xi_r - \E\Xi_r\right|^2\right)^{1/2}\,ds\,dr\right)                                     \\
		 & \leq 2\varepsilon \left\|b\right\|_{C^1} \int_0^t
		\left(1-e^{-(t-r)/\varepsilon}\right)
		\left( \left(\E \left|V_r^\varepsilon\right|^2\right)^{1/2} +1\right)\,dr                               \\
		 & \leq C \varepsilon t \,.
	\end{align*}
	Here, in the second inequality, we use that
	\begin{equation*}
		\left(\E \left|\Xi_r - \E \Xi_r\right|^2\right)^{1/2}\leq
		\left\|b\right\|_{C^1} \left(2 \left(\sum_k \E \left|V^{\varepsilon,k}_r\right|^2\right)^{1/2} + 2\right).
	\end{equation*}

	Finally, we use the previous two estimates in \eqref{ine:GWij} and sum over $j$.
	We have
	\begin{equation*}
		\begin{split}
			\sum_{j=1}^n \left|  \E {\Vtilal[,i]_0} Y^j_t  \right|
			 & \leq C \int_0^t \sum_{j=1}^n \left|\E \tilde{V}_0^{\varepsilon,i} Y^j_s\right|\,ds
			+ C_1 (\eps t^2 + \eps t + \eps)e^{C_2 t}   \,.
		\end{split}
	\end{equation*}
	By Gronwall's inequality,
	\begin{equation*}
		\sum_{j=1}^n \left|  \E {\Vtilal[,i]_0} Y^j_t  \right|
		\leq C_1 (\eps t^2 + \eps t + \eps + \varepsilon^2t)e^{C_2 t} \,,
	\end{equation*}
	which~\eqref{eq:average-approximate} follows immediately.
\end{proof}

\begin{proof}[Proof of Lemma~\ref{lem:average-approximate-2}]
	By Taylor expansion,
	\begin{align*}
		G(\Xal_t, t) - G(\E \Xal_t, t)
		 & = \grad_x G (\E \Xal_t, t)\cdot Y_t
		+ \frac{1}{2}  Y_t \cdot D_x^2 G(\eta_t, t)  Y_t
	\end{align*}
	where $\eta_t$ is a value between $\E \Xal_t$ and $\Xal_t$.
	Therefore,
	\begin{align*}
		 & \left| \E \left({\Vtilal[,i]_0}  \left( G(\Xal_t, t) - G(\E \Xal_t, t)\right)  \right)   \right| \\
		 & \leq \norm{G}_{C^1} \sum_{j=1}^n  \left|\E {\Vtilal[,i]_0}  Y^j_t \right|
		+ \frac{1}{2}\norm{G}_{C^2} \E \left( \left|{\Vtilal[,i]_0}\right| \left| Y_t \right|^2 \right)     \\
		 & \leq \left\|G\right\|_{C^1} C_1\eps(1+t)^2 e^{C_2t}
		+ \frac{1}{2} \left\|G\right\|_{C^2}
		\left(\E \left| \tilde{V}_0^{\varepsilon,i}\right|^2\right)^{1/2}
		\left(\E \left|Y_t\right|^4\right)^{1/2} \,.
	\end{align*}
	The last inequality follows Lemma~\ref{lem:average-approximate} and H\"{o}lder inequality.
	Using the estimate~\eqref{eq:Y4t} we obtain \eqref{eq:average-approximate-upgrade}
	as desired.
\end{proof}

\bibliographystyle{plain}
\bibliography{bibo}

\end{document}